\documentclass[11pt,a4paper]{article}

\usepackage{lipsum} % For \lipsum duh...
\usepackage{subfiles}
\usepackage{mathtools} % For \DeclarePairedDelimiterX
\usepackage{amsfonts} % For \mathbb
\usepackage{enumitem} % For enumerate environment with various labels like \arabic, \roman, \alph
\usepackage{graphicx} % For \includegraphics
\usepackage{epstopdf}

\DeclarePairedDelimiterX{\norm}[1]{\lVert}{\rVert}{#1}
\DeclarePairedDelimiterX{\inner}[2]{\langle}{\rangle}{\,#1,\,#2}
\DeclarePairedDelimiterX{\abs}[1]{\lvert}{\rvert}{#1}

\usepackage{geometry} % For better page layout/margins
\usepackage{authblk} % For \affil in the title
\usepackage{amsmath} % For \numberwithin to get equation numbering based on section numbering
\usepackage{amsthm} % For more theorem environments, for example TheoremNoNumber
\usepackage[pdfencoding=unicode,
            psdextra,
            breaklinks=true,
            colorlinks,
            linkcolor=blue]{hyperref} % pdfencoding and psdextra fixes rendering of math symbols in the bookmarks of the pdf
\usepackage[dvipsnames]{xcolor} % To get colors for use in \color{}, for example \color{teal} is unavailable without this.
\usepackage[T1]{fontenc} % To set good font encoding in the pdf so that non-English letters such as
\usepackage[swedish,english]{babel} % To make latex format the text properly for english and swedish
\usepackage{csquotes} % To prevent warning about missing csquotes.
\usepackage{silence}
\usepackage[backend=biber, style=numeric, maxnames=6]{biblatex}
\renewbibmacro{in:}{}
\AtEveryBibitem{%
	\clearfield{isbn}
	\clearfield{issn}
	\clearfield{number}
	\clearfield{volume}
	\clearfield{series}
	\clearfield{url}
	\clearfield{pages}
}

\DeclareFieldFormat*{title}{\mkbibemph{#1}}
\DeclareFieldFormat*{journaltitle}{#1}
\DeclareFieldFormat*{booktitle}{#1}

\renewbibmacro{in:}{%
  \ifentrytype{article}
    {}
    {\printtext{\bibstring{in}\intitlepunct}}}

\DeclareFieldFormat[article,periodical]{volume}{\mkbibbold{#1}}
\DeclareFieldFormat[article,periodical]{number}{\bibstring{number}~#1}
\renewbibmacro*{journal+issuetitle}{%
  \usebibmacro{journal}%
  \setunit*{\addspace}%
  \iffieldundef{series}
    {}
    {\newunit
     \printfield{series}%
     \setunit{\addspace}}%
  \printfield{volume}%
  \setunit{\addspace}%
  \usebibmacro{issue+date}%
  \setunit{\addcolon\space}%
  \usebibmacro{issue}%
  \setunit{\addcomma\space}%
  \printfield{number}%
  \setunit{\addcomma\space}%
  \printfield{eid}%
  \newunit}

\DeclareFieldFormat[article,periodical]{pages}{#1}

\DeclareFieldFormat{mrnumber}{%
  MR\addcolon\space
  \ifhyperref
    {\href{http://www.ams.org/mathscinet-getitem?mr=#1}{\nolinkurl{#1}}}
    {\nolinkurl{#1}}}

\renewbibmacro*{doi+eprint+url}{%
  \iftoggle{bbx:doi}
    {\printfield{doi}}
    {}%
  \newunit\newblock
  \printfield{mrnumber}%
  \newunit\newblock
  \iftoggle{bbx:eprint}
    {\usebibmacro{eprint}}
    {}%
  \newunit\newblock
  \iftoggle{bbx:url}
    {\usebibmacro{url+urldate}}
    {}}

\newtheorem{theorem}{Theorem}[section]

\newtheorem{lemma}[theorem]{Lemma}

\theoremstyle{definition}
\newtheorem{definition}[theorem]{Definition}

\newtheorem{remark}[theorem]{Remark}

\numberwithin{equation}{section}

\DeclarePairedDelimiterX{\seminorm}[1]{[}{]}{#1}

\begin{document}

\title{Corrosion detection by identification of a nonlinear Robin boundary condition} 

\author{David Johansson}
\affil{Department of Mathematics, Aarhus University\\\texttt{johansson@math.au.dk}}

\date{January 8, 2026}
\maketitle

\begin{abstract}
We study an inverse boundary value problem in corrosion detection.
The model is based on a conductivity equation with nonlinear Robin boundary condition.
We prove that the nonlinear Robin term can be identified locally from Cauchy data measurements on a subset of the boundary.
A possible strategy for turning a local identification result into a global one is suggested, and a partial result is proved in this direction.
The inversion method is an adaptation to this nonlinear Robin problem of a method originally developed for semilinear elliptic equations.
The strategy is based on linearization and relies on parametrizing solutions of the nonlinear equation on solutions of the linearized equation.\\
\texttt{MSC 2020:} 35R30(Primary), 35J25, 35J65\\
\texttt{Keywords:} Corrosion detection, Inverse Robin problem, Nonlinear Robin boundary condition
\end{abstract}

\section{Introduction}
We investigate the identifiability of corrosion on an object from voltage and current measurements on the uncorroded part of the object.
The mathematical model is based on a differential equation for the electric potential in the object.
If there is corrosion on a subset of the boundary, this appears as a nonlinear Robin boundary condition.
The inverse Robin problem was introduced in \cite{inglese97} as a tool for corrosion detection, with a linear Robin condition.
In the present article we study the identifiability of corrosion using a nonlinear Robin condition motivated by modeling done in \cite{vx98}.
The Neumann-Robin model we use here corresponds to an experimental setup where a current is prescribed, and the induced voltage is measured.
One could also prescribe the voltage and measure the current and this would lead to a mixed Dirichlet-Robin problem.
In order to simplify the analysis a bit, we only consider the Neumann-Robin problem.
Similar results should hold also for the Dirichlet-Robin problem.

Let \( \Omega\subset\mathbb{R}^{n} \) be an open set representing the possibly corroded object.
The boundary \( \partial\Omega \) is partitioned into an \emph{accessible} region \( \Gamma_{A} \) and an \emph{inaccessible} region \( \Gamma_{I} \).
A current \( f \) is prescribed on \( \Gamma_{A} \) and any corrosion occurs on \( \Gamma_{I} \).
The conductivitity thoughout \( \Omega \) is denoted by \( \gamma \).
For these quantities, we use the following running assumptions
\begin{enumerate}[label=(\roman*)]
	\item\label{assumption-1} \( \Omega\subset \mathbb{R}^{n} \) is a bounded connected open set with \( C^{1} \) boundary \( \partial\Omega \)
	\item\label{assumption-2} \( \partial\Omega = \overline{\Gamma_{I}\cup \Gamma_{A}} \) for \( \Gamma_{I}, \Gamma_{A}\subset \partial\Omega \) relatively open and nonempty with \( \Gamma_{I}\cap \Gamma_{A} = \emptyset \)
	\item\label{assumption-3} \( \gamma\in C^{0,1}(\overline{\Omega}; \mathbb{R}^{n\times n}) \) is symmetric with eigenvalues uniformly bounded below on \( \overline{\Omega} \), that is, there exists \( C>0 \) such that \( \xi^{t} \gamma(x)\xi\geq C \abs{\xi}^{2} \) for all \( x\in\overline{\Omega} \), \( \xi\in\mathbb{R}^{n} \)
\end{enumerate}
The electric potential is denoted by \( u \) and \( \partial_{\nu,\gamma}u \coloneqq (\gamma\nabla u)\cdot \nu \) denotes the conormal derivative on \( \partial\Omega \) where \( \nu \) is the outward pointing unit normal vector.
The governing equations for the electric potential \( u \) are the following,
\begin{equation}\label{eq-nonlinear}
\begin{cases}
	-\operatorname{div}(\gamma\nabla u) = 0&\text{in }\Omega, \\
	\partial_{\nu,\gamma}u = f&\text{on }\Gamma_{A}, \\
	\partial_{\nu,\gamma}u + a(x,u) = 0&\text{on }\Gamma_{I}.
\end{cases}
\end{equation}
We prove that voltage-current pairs \( (u\vert_{\Gamma_{A}}, f) \) on the accessible boundary region \( \Gamma_{A} \) uniquely determine the term \( a(x,u) \) locally on the inaccessible boundary region \( \Gamma_{I} \).
If the function \( a \) is nonzero then there is corrosion and if the function is identically zero then there is no corrosion.
In particular the function \( a \) is allowed to be zero, meaning that we can detect the presence of corrosion as well as the absence of corrosion.

\subsection*{Notation and conventions}
For any real number \( p>1 \) its Hölder conjugate is denoted \( p' \), \( \frac{1}{p} + \frac{1}{p'} = 1 \).
The symbol \( P_{X} \) denotes a projection onto the Banach space \( X \).
The symbol \( C \) denote constants, and may denote several different constants in the same proof whenever the precise value of the constant is insignificant for the proof.

\subsection{Precise problem formulation and main results}\label{sec-problem-def}
The definition of the parameter-to-measurement map is central to many inverse problem, as this is the map we wish to invert.
More specifically, we study the injectivity of this map.
A typical measurement for linear elliptic partial differential equations is the Cauchy data set, which is guaranteed to exist for a large family of coefficients by the Fredholm alternative.
Generalizing this to nonlinear equations turns out to be somewhat delicate.
On the one hand the measurements need to be well-defined, which cannot be settled by the Fredholm alternative, and on the other hand the measurements need to be useful.
In the nonlinear setting, a natural idea is to approach the problem by local arguments based on linearization.
This might suggest that also the measurements should somehow be local in order to be useful.
But for the more established theory of linear equations the Cauchy data is always global, so it is not immediately clear in what sense the Cauchy data should be local.
The goal of this subsection is to make sense of both a local and a global version of the inverse nonlinear Robin problem. 
The section ends with statements of the main results.

First, we define a set of \emph{admissible} nonlinearities for which we solve the inverse problem.
This could be approached in at least two ways.
One could impose structural assumptions to ensure the equation is well-posed and then solve the inverse problem in this subset of nonlinearities.
Such structural assumptions could be the positivity of the derivative \( \partial_{z}a(x,z)>c\geq 0 \), as in \cite{is94}.
This is not the approach taken in the present article.
Instead, we try to make minimal assumptions by simply assuming that the nonlinearity \( a(x,z) \) must allow the equation to have at least one solution.
\begin{definition}[Admissible nonlinearity]
Given a fixed \( p>n \) and its Hölder conjugate \( p' \), any function \( a\in C_{loc}^{1,1}(\mathbb{R}, L^{\infty}(\Gamma_{I})) \) is said to be an \emph{admissible nonlinearity} if and only if there exist \( f\in W^{1-1/p',p'}(\partial\Omega)^{*} \) with \( \operatorname{supp}(f)\subseteq \Gamma_{A} \) and \( u\in W^{1,p}(\Omega) \) satisfying \eqref{eq-nonlinear}.
The set of all admissible nonlinearities is denoted \( \mathcal{A}_{ad}\subseteq C_{loc}^{1,1}(\mathbb{R},L^{\infty}(\Gamma_{I})) \).
\end{definition}
\begin{remark}
	The integrability exponent \( p \) is fixed for the class of admissible nonlinearities \( \mathcal{A}_{ad} \).
\end{remark}
Note that \( \mathcal{A}_{ad} \) is nonempty.
Any function of the form \( a(x,z) = q(x)z^{k} \) for \( k\in\mathbb{N} \) belongs to \( \mathcal{A}_{ad} \) and, more generally, any function \( a(x,z) \) which is locally \( C^{1,1} \) in \( z \) and \( L^{\infty} \) in \( x \) and whose derivative vanish at zero, \( \partial_{z}a(x,0) = 0 \).
For these examples of nonlinearities, \( u(x)=0 \) is a solution.

For some fixed \( p>n \) and any \( a\in\mathcal{A}_{ad} \) there exist solutions in \( W^{1,p}(\Omega) \), meaning that Cauchy data measurements are possible.
In principle the \emph{global} Cauchy data set \( C_{a} \) can be defined as the Cauchy data of all solutions,
\begin{equation*}
\begin{aligned}
	C_{a} \coloneqq \{(u\vert_{\Gamma_{A}},\partial_{\nu,\gamma}u\vert_{\Gamma_{A}})\colon &u\in W^{1,p}(\Omega) \text{ solves \eqref{eq-nonlinear} for some }f\in W^{1-1/p',p'}(\partial\Omega)^{*}\\&\text{with }\operatorname{supp}(f)\subseteq\Gamma_{A} \}.
\end{aligned}
\end{equation*}
For a precise definition of the conormal derivative \( \partial_{\nu,\gamma}u\in W^{1-1/p',p'}(\partial\Omega)^{*} \) and its restriction \( \partial_{\nu,\gamma}u\vert_{\Gamma_{A}}\in W^{1-1/p',p'}_{0}(\partial\Omega,\Gamma_{A})^{*} \), see Lemma \ref{lemma-conormal} and equations \eqref{eq-functional-restriction}, \eqref{eq-traces-space-support} in the appendix.
The global Cauchy data \( C_{a} \) is a natural candidate for the parameter-to-measurement map, as follows:
\begin{equation}\label{eq-parameter-to-measurement-1}
	\mathcal{A}_{ad}\ni a\mapsto C_{a}\subseteq W^{1-1/p,p}(\Gamma_{A})\times W^{1-1/p',p'}_{0}(\partial\Omega,\Gamma_{A})^{*}.
\end{equation}
We may then consider the inverse problem of determining the injectivity of this map:
\begin{center}
\begin{minipage}{0.8\textwidth}
\begin{itemize}
	\item[{\bfseries IP1:}]\label{ip1-1} For \( a_{1},a_{2} \in \mathcal{A}_{ad} \), it is true that \( C_{a_{1}}=C_{a_{2}}\implies a_{1}=a_{2} \)?
\end{itemize}
\end{minipage}
\end{center}
If this map is injective then the Cauchy data set \( C_{a} \) determines \( a \) in \( \Gamma_{I}\times\mathbb{R} \).
In other words, injectivity of the map \eqref{eq-parameter-to-measurement-1} leads to global identification of \( a \).

The technique from \cite{jns23} that we use to solve the inverse problems is local in nature, and uses a form of local Cauchy data set.
The \emph{local Cauchy data} set \( C_{a}^{w,\delta} \) is defined to be the Cauchy data which is contained in a \( \delta \)-ball around the Cauchy data of some solution \( w \), as follows
\begin{equation*}
\begin{aligned}
	C_{a}^{w,\delta} \coloneqq \{(u\vert_{\Gamma_{A}},\partial_{\nu,\gamma}u\vert_{\Gamma_{A}})\colon &u\in W^{1,p}(\Omega) \text{ solves \eqref{eq-nonlinear} for some }f\in W^{1-1/p',p'}(\partial\Omega)^{*} \\
	& \text{with }\operatorname{supp}(f)\subseteq\Gamma_{A}\text{, }\norm{u\vert_{\Gamma_{A}}-w\vert_{\Gamma_{A}}}_{W^{1,p}(\Gamma_{A})}<\delta, \\
	&\text{ and }\norm{\partial_{\nu,\gamma}u\vert_{\Gamma_{A}}-\partial_{\nu,\gamma}w\vert_{\Gamma_{A}}}_{W^{1-1/p',p'}(\partial\Omega,\Gamma_{A})^{*}}<\delta\}.
\end{aligned}
\end{equation*}
Whenever \( a\in\mathcal{A}_{ad} \) then there exists at least one solution \( w \) and Lemma \ref{lemma-solution-map} implies that there exists a neighborhood of solutions around \( w \) whose Cauchy data belong to \( C_{a}^{w,\delta} \).
In other words, \( C_{a}^{w,\delta} \) contains measurements of infinitely many solutions.
\begin{remark}
This definition of local Cauchy data is somewhat different than what is used other works on elliptic equations of semilinear type, such as \cite{jns23}.
In the present article, the \( \delta \)-neighborhood in the definition of \( C_{a}^{w,\delta} \) is defined with norms on the boundary \( \partial\Omega \), while in \cite{jns23} it is defined with a norm in the interior of \( \Omega \).
It is desirable to use norms on the boundary \( \partial\Omega \) since we wish to make boundary measurements.
But the \( \delta \)-neighborhood is mainly used to ensure that the Cauchy data of the solutions produced by the implicit function theorem belongs to \( C_{a}^{w,\delta} \).
For the equation treated in \cite{jns23} this is more difficult to do.
Hence the use of different norms.
\end{remark}
The local version of \hyperref[ip1-1]{IP1} is formulated as follows:
\begin{center}
\begin{minipage}{0.8\textwidth}
\begin{itemize}
	\item[{\bfseries IP2:}]\label{ip2-1} For \( a_{1},a_{2}\in \mathcal{A}_{ad} \) is it true that for some \( \delta>0 \)
\begin{equation*}
	C_{a_{1}}^{w_{1},\delta} = C_{a_{2}}^{w_{2},\delta}\implies a_{1}=a_{2}\text{ locally},
\end{equation*}
where \( w_{1},w_{2} \) are solutions for \( a_{1},a_{2} \), respectively, with equal Cauchy data?
\end{itemize}
\end{minipage}
\end{center}
The first main result of the present article is to adapt the method developed in \cite{jns23} to the nonlinear Robin problem.
This provides an answer in the affirmative to \hyperref[ip2-1]{IP2}.
\begin{remark}
The local problem \hyperref[ip2-1]{IP2} does not correspond to injectivity of a map of the form \( a\mapsto C_{a}^{w,\delta} \), since injectivity is equivalent to global identification.
For the identification stated in \hyperref[ip2-1]{IP2} to be global we would need \( a\in\mathcal{A}_{ad} \) to only be defined on the subset on which the local identification holds, but this subset is a priori not known.
\end{remark}
The second main result is related to the global inverse problem \hyperref[ip1-1]{IP1}.
This part of the article adds no new technical results, and is mainly intended to highlight the role of the solutions of the equation when determining nonlinear terms.
In particular it shows that insights about the the behaviour of solutions may lead to global determination and therefore a solution of \hyperref[ip1-1]{IP1}.

The key observation is that the method from \cite{jns23} used to solve \hyperref[ip2-1]{IP2} essentially identifies \( a_{1}(x,z) = a_{2}(x,z) \) at points \( x,z \) such that \( z = u(x) \) for some solution \( u \) with Cauchy data in \( C_{a_{1}}^{w_{1},\delta} \).
In other words, the nonlinearities are identified on the graphs of solutions having Cauchy data in the local Cauchy data set.
If the global condition \( C_{a_{1}}=C_{a_{2}} \) implies the local Cauchy data condition \( C_{a_{1}}^{w_{1},\delta}= C_{a_{2}}^{w_{2},\delta} \) for arbitrary \( w_{1},w_{2} \) that have equal Cauchy data, then the same method from \cite{jns23} identifies the nonlinearities on the union of the graphs of all solutions.
The union of all such graphs is the so-called \emph{reachable set},
\begin{equation*}
	E_{a}\coloneqq \{(x,u(x))\in\Gamma_{I}\times\mathbb{R}\colon \text{for any solution } u \text{ with Cauchy data in } C_{a}\}.
\end{equation*}
If one can characterize the reachable set \( E_{a} \) and in particular show that \( E_{a}=\Gamma_{I}\times\mathbb{R} \) then one would prove that \( a_{1}=a_{2} \) on \( \Gamma_{I}\times\mathbb{R} \).
We may therefore solve \hyperref[ip1-1]{IP1} by solving the following two problems,
\begin{center}
\begin{minipage}{0.8\textwidth}
\begin{itemize}
	\item[{\bfseries IP3:}]\label{ip3-1} Does the Cauchy data \( C_{a} \) uniquely determine the nonlinearity \( a\in\mathcal{A}_{ad} \) on the reachable set \( E_{a} \)?
	\item[{\bfseries IP4:}]\label{ip4-1} Characterize the reachable set \( E_{a} \).
\end{itemize}
\end{minipage}
\end{center}
The second main result is a partial result on \hyperref[ip4-1]{IP4}.
We show that \( E_{a} \) is an open subset of \( \Gamma_{I}\times\mathbb{R} \), using a minor modification of the Runge approximation argument from \cite{jns23}.

The following theorem provides an answer in the affirmative to \hyperref[ip2-1]{IP2}.
\begin{theorem}\label{thm-main-result-1}
Let \( n\geq 2 \) and \( \Omega,\partial\Omega,\Gamma_{A},\Gamma_{I} \) satisfy \ref{assumption-1} and \ref{assumption-2}.
Assume that \( \gamma \) satisfies \ref{assumption-3} and \( a_{1},a_{2}\in\mathcal{A}_{ad} \).
If there exists solutions \( w_{1},w_{2} \) corresponding to \( a_{1},a_{2} \) with equal Cauchy data on \( \Gamma_{A} \) and constants \( \delta,C>0 \) such that
\begin{equation*}
	C_{a_{1}}^{w_{1},\delta} = C_{a_{2}}^{w_{2},\delta}
\end{equation*}
then \( w_{1}=w_{2} \) on \( \Gamma_{I} \) and it holds, for some \( \lambda>0 \), that
\begin{equation*}
	a_{1}(x,w_{1}(x)+\varepsilon) = a_{2}(x,w_{1}(x)+\varepsilon)\quad \text{for a.e. } x\in\Gamma_{I},\,\,\abs{\varepsilon}\leq\lambda.
\end{equation*}
\end{theorem}
The following theorem resolves \hyperref[ip3-1]{IP3} and provides a partial result for \hyperref[ip4-1]{IP4}.
\begin{theorem}\label{thm-main-result-2}
Let \( n\geq 2 \) and \( \Omega,\partial\Omega,\Gamma_{A},\Gamma_{I} \) satisfy \ref{assumption-1} and \ref{assumption-2}.
Assume that \( \gamma \) satisfies \ref{assumption-3} and \( a_{1},a_{2}\in\mathcal{A}_{ad} \).
Then
\begin{equation*}
	C_{a_{1}} = C_{a_{2}} \implies
\begin{cases}
	a_{1}(x,z) = a_{2}(x,z)\quad \text{for a.e. } (x,z)\in E_{a_{1}},\\
	E_{a_{1}} = E_{a_{2}} \text{ is an open subset of }\Gamma_{I}\times\mathbb{R}.
\end{cases}
\end{equation*}
\end{theorem}

The article is organized as follows.
In Section \ref{sec-solution-map} we prove the solvability results for the nonlinear equation, which are used both to justify the definition of the local Cauchy data set and to solve the inverse problem.
In Section \ref{sec-inverse-problem} we prove the main results, Theorem \ref{thm-main-result-1} and Theorem \ref{thm-main-result-2}.
In Appendix \ref{sec-appendix}, we establish for the sake of completeness the necessary theory for weak solutions of the linear equation.

\subsection{Historical overview}
Some of the earliest attempts at studying corrosion detection in the framework of inverse problems appear in \cite{ksv96} and \cite{ks95}.
Corrosion leads to a loss of mass, which in turn leads to a change of shape.
Starting with a rectangular object with known dimensions they determine where the shape deviates from a ractangle, thus identifying the corrosion.
A new model is proposed in \cite{vx98}, where the object has a known fixed shape and the presence of corrosion is modelled by boundary conditions.
If the corrosion is electrochemical, an electric current will flow through the corroded area.
Arguing based on Faraday's law and electrochemical principles, the authors conclude that this current is dependent on the electric potential in a nonlinear way.
The result is a nonlinear Robin boundary condition for the electric potential.
This is said to apply to electrochemical corrosion, which the authors of \cite{vx98} claim covers most corrosion.

The formulation of an inverse Robin problem for corrosion detection was first given in \cite{inglese97} and the first identification result was proved in \cite{cj99}, where a smooth coefficient in a linear Robin condition was identified from one boundary measurement.
The main ingredient in the proof is unique continuation from Cauchy data.
A very similar proof is applied to a nonlinear Robin condition in \cite{fi06}, where the nonlinearity is real-analytic in the solution and independent of space.
The result is limited to a 2-dimensional square, but uses only a single boundary measurement.
The identifiability of an \( L^{\infty} \) coefficient from a single measurement in the linear Robin problem was studied in \cite{bbl16}.
They show that the coefficient is identifiable in the 2-dimensional case and provide a counterexample in 3 dimensions.
It was shown in \cite{cr21}, however, that the Dirichlet-to-Neumann map uniquely determines an \( L^{\infty} \) coefficient in the linear Robin condition.
This was proved using an integral identity and density of solutions, similar to the classical Calderón problem, but with the density coming from Runge approximation rather than specialized families of solutions like complex geometric optics solutions.

The Cafarelli-Silvestre extension provides a connection between the linear Robin problem in a flat domain and the fractional Schrödinger equation on the boundary.
The inverse linear Robin problem is therefore related to inverse problems for the fractional Schrödinger equation, such as the ones studied in \cite{cmr21}.
The strong uniqueness result in \cite{cj99} for the inverse linear Robin problem using a single boundary measurement has counterparts for the nonlocal fractional Schrödinger equation \cite{grsu20}.

While the present article is focused on the question of identifiability, there are several interesting variations of the problem.
For the inverse linear Robin problem with one measurements, logarithmic stability estimates were established in \cite{adpr03} and \cite{cj16}.
Similar logarithmic stability results are proved in \cite{bcd16} for an inverse Robin-type problem for the stationary Navier-Stokes' equations.
With measurements given by the Neumann-to-Dirichlet map, Lipschitz stability is established for an inverse transmission-type Robin problem in \cite{hm19}.
In \cite{lf05} the inverse linear Robin problem is reformulated in terms of integral equations on the boundary, and in \cite{bm15} this integral equation approach is used to solve an inverse problem for a transmission-type Robin problem.

In the present article, we use ideas coming from the Calderón problem for a semilinear Schrödinger-type equation to solve the inverse nonlinear Robin problem.
The original idea comes from \cite{is94}, where a linearization strategy is used to reduce the inverse problem for the semilinear equation to the corresponding problem for the linearized equation.
While their proof does not require much smoothness out of the nonlinear term, it does require the linearized equation to satisfy a maximum principle.
Recently it was noticed that one can identify higher order derivatives by solving inverse problems for higher order linearizations.
In this way, the full Taylor expansion can be reconstructed, see for example \cite{llls20}.
The higher order linearization method has proved to be quite flexible and has been applied to a wide range of equations, but has the drawback that the nonlinear term to be identified must be real-analytic.

The method in \cite{is94} was generalized to ill-posed equations in \cite{jns23}. In the present article, we adapt the method in \cite{jns23} to the inverse nonlinear Robin problem.
In this case, however, it is not necessary to solve an inverse problem for the linearized equation.
Instead it is enough to apply unique continuation to the nonlinear equation similar to \cite{fi06} and then apply Runge approximation to the linearized equation.
When using Runge approximation on the linear equation, the goal is to gain some information about the solutions of the nonlinear equation.
The link to pass information from the linear to nonlinear problem is a parametrization of the solutions of the nonlinear equation on solutions of the linearized equation, as in \cite{jns23}.

\subsection*{Acknowledgements}
This work is supported by the Research Council of Finland (Centre of Excellence in Inverse Modelling and Imaging and FAME Flagship, 312121 and 359208).
The author would like to thank Professor Mikko Salo for many helpful discussions and in particular for helping to navigate the vast fields of PDEs and inverse problems.

\section{Solvability of the nonlinear equation}\label{sec-solution-map}
In this section we prove solvability results for the nonlinear equation \eqref{eq-nonlinear}, and the purpose of these results is twofold.
On the one hand the results ensure that the Cauchy data measurements contain enough information.
On the other hand the results provide the link between linear and nonlinear theory which plays a crucial role in the identification proof.
We consider weak solutions of \eqref{eq-nonlinear}.
More precisely, \( u\in W^{1,p}(\Omega) \) is said to be a weak solution of \eqref{eq-nonlinear} if and only if it satisfies
\begin{equation*}
	\int_{\Omega}\gamma\nabla u\cdot \nabla \varphi\,dx + \int_{\Gamma_{I}}a(x,u)\varphi\,ds = \int_{\Gamma_{A}}f\varphi\,ds\quad\forall \varphi\in W^{1,p'}(\Omega).
\end{equation*}
This is justified in the usual way using integration by parts.
Note that the composition \( \Gamma_{I}\ni x\mapsto a(x,u(x)) \) is in \( L^{\infty}(\Gamma_{I}) \) when \( a\in\mathcal{A}_{ad} \) and \( u\in W^{1,p}(\Omega) \) with \( p>n \).
Measurability follows from Theorem 4.51 in \cite{ab06} since \( a \) is a Carathéodory function, and it is bounded because the range of \( u \) is bounded and \( a \) is uniformly bounded on bounded subsets of \( \mathbb{R} \).
The above integral over \( \Gamma_{I} \) is therefore well-defined.

The kernel of the following linear equation appear throughout the article,
\begin{equation}\label{eq-linear-kernel}
\begin{cases}
	-\operatorname{div}(\gamma\nabla v) = 0&\text{in }\Omega, \\
	\partial_{\nu,\gamma}v = f&\text{on }\Gamma_{A}, \\
	\partial_{\nu,\gamma}v + qv = 0&\text{on }\Gamma_{I}.
\end{cases}
\end{equation}
We use both the kernel \( \mathcal{N} \) and the trace space of the kernel \( \mathcal{N}_{\Gamma_{A}} \) defined as follows
\begin{equation*}
\begin{aligned}
	\mathcal{N} &\coloneqq \{v\in W^{1,p}(\Omega)\colon v \text{ is a weak solution of \eqref{eq-linear-kernel} with } f = 0\}, \\
	\mathcal{N}_{\Gamma_{A}} &\coloneqq \{v\vert_{\Gamma_{A}}\colon v\in \mathcal{N} \}.
\end{aligned}
\end{equation*}
The coefficient \( q \) used in the Robin condition in the definitions of these spaces is always an \( L^{\infty}(\Gamma_{I}) \) function, and everywhere in this section except Lemma \ref{lemma-forced-well-posedness} it is given explicitly by \( \partial_{u}a(x,w) \) where \( w \) is a function known from the context.

The following analysis relies on the decompositions
\begin{equation}\label{eq-decompositions}
\begin{aligned}
	L^{p}(\Omega) &= \mathcal{N}\oplus\mathcal{N}^{\perp}, \\
	L^{p}(\Gamma_{A}) &= \mathcal{N}_{\Gamma_{A}}\oplus\mathcal{N}_{\Gamma_{A}}^{\perp_{L}},\\
	W^{1-1/p',p'}_{0}(\partial\Omega,\Gamma_{A})^{*} &= \mathcal{N}_{\Gamma_{A}} \oplus \mathcal{N}_{\Gamma_{A}}^{\perp_{W}}, \\
\end{aligned}
\end{equation}
where \( \mathcal{N}_{\Gamma_{A}}^{\perp_{L}} \) and \( \mathcal{N}_{\Gamma_{A}}^{\perp_{W}} \) denotes the subspaces of \( L^{p}(\Gamma_{A}) \) and \( W^{1-1/p',p'}_{0}(\partial\Omega,\Gamma_{A})^{*} \), respectively, that annihiliate \( \mathcal{N}_{\Gamma_{A}} \) (For a definition of \( W^{1-1/p',p'}_{0}(\partial\Omega,\Gamma_{A})^{*} \), see equations \eqref{eq-functional-restriction}, \eqref{eq-traces-space-support} in the appendix).
In the remainder of the article, we write \( \mathcal{N}_{\Gamma_{A}}^{\perp} \) for both of these spaces, as it is clear from the context when the objects are functions or functionals.
Corresponding to each decomposition, we have projections
\begin{equation}\label{eq-projections}
\begin{aligned}
	P_{\mathcal{N}}&\colon L^{p}(\Omega)\to \mathcal{N}, \\
	P_{\mathcal{N}_{\Gamma_{A}}}&\colon L^{p}(\Gamma_{A})\to \mathcal{N}_{\Gamma_{A}}, \\
	P_{\mathcal{N}_{\Gamma_{A}}^{*}}&\colon W_{0}^{1-1/p',p'}(\partial\Omega,\Gamma_{A})^{*}\to \mathcal{N}_{\Gamma_{A}}.
\end{aligned}
\end{equation}
The decompositions \eqref{eq-decompositions}, as well as explicit formulas for the projections, are established in Lemma \ref{lemma-kernel-projection}.

While the linear Robin problem is not assumed to be well-posed, we do require some notion of well-posedness to hold.
If one restricts attention to a suitable subspace then one can establish a sufficiently strong well-posedness result.
This is the content of the next lemma.
Using the well-posedness in this subspace we are able to solve a nonlinear fixed point equation in Lemma \ref{lemma-fixed-point}.
The solution of this fixed point equation is then used to construct solutions of the nonlinear Robin problem in Lemma \ref{lemma-solution-map}.
\begin{lemma}\label{lemma-forced-well-posedness}
Let \( 1<p<\infty \), \( n\geq 2 \) and \( \Omega,\partial\Omega,\Gamma_{A},\Gamma_{I} \) satisfy \ref{assumption-1} and \ref{assumption-2}.
Assume that \( \gamma \) satisfies \ref{assumption-3} and let \( q\in L^{\infty}(\Gamma_{I}) \).
For any \( f,g\in W^{1-1/p',p'}(\partial\Omega)^{*} \) satisfying \( \operatorname{supp}(f)\subseteq\Gamma_{A} \), \( \operatorname{supp}(g)\subseteq\Gamma_{I} \) there exist a unique \( \Phi=\Phi(f,g)\in N_{\Gamma_{A}} \) given by
\begin{equation}\label{eq-forced-wellposedness-2}
	\Phi(f,g) \coloneqq P_{\mathcal{N}_{\Gamma_{A}}^{*}}(f+g)
\end{equation}
such that the equation
\begin{equation*}
\begin{cases}
	-\operatorname{div}(\gamma\nabla u) = 0&\text{in }\Omega, \\
	\partial_{\nu,\gamma}u = f-\Phi&\text{on }\Gamma_{A}, \\
	\partial_{\nu,\gamma}u + qu = g&\text{on }\Gamma_{I},
\end{cases}
\end{equation*}
admits weak solutions in \( W^{1,p}(\Omega) \).
Moreover, there is a unique solution whose trace annihilates the kernel, \( u\vert_{\Gamma_{A}}\in\mathcal{N}_{\Gamma_{A}}^{\perp} \), and this solution satisfies
\begin{equation}\label{eq-forced-wellposedness-3}
	\norm{u}_{W^{1,p}(\Omega)} \leq C(\norm{f}_{W^{1-1/p',p'}(\partial\Omega)^{*}} + \norm{g}_{W^{1-1/p',p'}(\partial\Omega)^{*}}).
\end{equation}
\end{lemma}
\begin{proof}
By Lemma \ref{lemma-linear-solvability}, a solution exists if and only if the functional
\begin{equation}\label{eq-forced-wellposedness-1}
	F(\varphi) \coloneqq \inner{f}{\varphi} - \int_{\Gamma_{A}}\Phi(f,g)\varphi\,ds + \inner{g}{\varphi}
\end{equation}
annihilates the kernel \( \mathcal{N} \).
Let \( \psi_{1},\ldots,\psi_{N} \) be a basis of \( \mathcal{N} \) as in Lemma \ref{lemma-kernel-basis}, with traces onto \( \Gamma_{A} \) being an \( L^{2}(\Gamma_{A}) \)-orthonormal basis of \( \mathcal{N}_{\Gamma_{A}} \).
By integrating \( \Phi \), as given in \eqref{eq-forced-wellposedness-2}, against an arbitrary \( \varphi =\Sigma_{i=1}^{N} \alpha_{i}\psi_{i}\in\mathcal{N} \) over \( \Gamma_{A} \) and using the \( L^{2}(\Gamma_{A}) \)-orthonormality of \( \psi_{i} \), we get
\begin{equation*}
	\int_{\Gamma_{A}}\Phi(f,g)\varphi\,ds = \inner{f}{\varphi} + \inner{g}{\varphi}.
\end{equation*}
In other words, the functional \( F \)  in \eqref{eq-forced-wellposedness-1} annihilates \( \mathcal{N} \).
To see uniqueness of \( \Phi\in\mathcal{N}_{\Gamma_{A}} \), suppose there are two \( \Phi_{1},\Phi_{2} \) and denote by \( F_{1},F_{2} \) their corresponding functionals defined as in \eqref{eq-forced-wellposedness-1}.
Since \( \Phi_{2}-\Phi_{1}\in\mathcal{N}_{\Gamma_{A}} \), there exist a \( \varphi\in\mathcal{N} \) with \( \varphi\vert_{\Gamma_{A}} = \Phi_{2}-\Phi_{1} \).
Since both \( F_{1} \) and \( F_{2} \) annihilate \( \mathcal{N} \) it follows that
\begin{equation*}
	0 = F_{1}(\varphi) - F_{2}(\varphi) = \norm{\Phi_{2}-\Phi_{1}}_{L^{2}(\Gamma_{A})}^{2}
\end{equation*}
which shows that \( \Phi_{1}=\Phi_{2} \).
The remaining properties of the solution follow straight from Lemma \ref{lemma-linear-solvability}.
The estimate \eqref{eq-forced-wellposedness-3} follows from the same lemma, after noting that the functional \( H\in W^{1-1/p',p'}(\partial\Omega)^{*} \) induced by \( \Phi \), \( H(\varphi)\coloneqq \int_{\Gamma_{A}}\Phi\varphi\,ds \), satisfies
\begin{align*}
	\norm{H}_{W^{1-1/p',p'}(\partial\Omega)^{*}} &\leq \norm{\Phi}_{L^{p}(\Gamma_{A})} =\norm{P_{\mathcal{N}_{\Gamma_{A}}^{*}}(f+g)}_{L^{p}(\Gamma_{A})} \\
	&\leq C(\norm{f}_{W^{1-1/p',p'}(\partial\Omega)^{*}}+\norm{g}_{W^{1-1/p',p'}(\partial\Omega)^{*}}).\qedhere
\end{align*}
\end{proof}

\begin{lemma}[Fixed point equation]\label{lemma-fixed-point}
Let \( n\geq 2 \), \( p>n \), and \( \Omega,\partial\Omega,\Gamma_{A},\Gamma_{I} \) satisfy \ref{assumption-1} and \ref{assumption-2}.
Assume that \( \gamma \) satisfies \ref{assumption-3} and \( a\in\mathcal{A}_{ad} \).
Define \( R\colon W^{1,p}(\Omega)\to L^{p}(\Gamma_{I}) \) by
\begin{equation*}
	R(h)(x) \coloneqq \int_{0}^{1}[\partial_{u}a(x,w(x)+th(x))-\partial_{u}a(x,w(x))]h(x)\,dt,\quad x\in\Gamma_{I}
\end{equation*}
and let \( \Phi(f,g) \) be given by \eqref{eq-forced-wellposedness-2}.
Assume that \( w,v \in W^{1,p}(\Omega) \) and \( \norm{v}_{W^{1,p}(\Omega)}\leq \delta \) for \( \delta>0 \) sufficiently small.
Then there exists a unique solution \( r\in W^{1,p}(\Omega) \) of
\begin{equation}\label{eq-fixed-point}
\begin{cases}
	-\operatorname{div}(\gamma\nabla r) = 0&\text{in }\Omega, \\
	\partial_{\nu,\gamma}r = \Phi(0,R(v+r))&\text{on }\Gamma_{A}, \\
	\partial_{\nu,\gamma}r + \partial_{u}a(x,w)r = -R(v+r)&\text{on }\Gamma_{I},
\end{cases}
\end{equation}
satisfying \( r\vert_{\Gamma_{A}}\in\mathcal{N}^{\perp}_{\Gamma_{A}} \), \( \norm{r}_{W^{1,p}(\Omega)}\leq \delta \) and
\begin{equation}\label{eq-fixed-point-quadratic}
	\norm{r}_{W^{1,p}(\Omega)}\leq C\norm{v}_{W^{1,p}(\Omega)}^{2}.
\end{equation}
The map \( W^{1,p}(\Omega)\ni v\mapsto r\in W^{1,p}(\Omega) \) obtained in this manner is Lipschitz continuous.
\end{lemma}
\begin{proof}
Let \( G\colon W^{1-1/p',p'}(\partial\Omega)^{*}\times W^{1-1/p',p'}(\partial\Omega)^{*} \to W^{1,p}(\Omega) \) be the bounded solution map from Lemma \ref{lemma-forced-well-posedness} which maps \( f,g \) to a solution \( u = G(f,g) \) with \( u\vert_{\Gamma_{A}}\in\mathcal{N}_{\Gamma_{A}}^{\perp} \) corresponding to the boundary data \( f-\Phi(f,g) \) and \( g \).
For a fixed function \( v\in W^{1,p}(\Omega) \) define the map \( T\colon W^{1,p}(\Omega)\to W^{1,p}(\Omega) \) by \( T_{v}(r)\coloneqq G(0,-R(v+r)) \).
We prove existence of the solution \( r \) by proving existence of a fixed point for \( T \), \( r = T_{v}(r) \)

First, we show that \( T \) maps a small \( \delta \)-ball in \( W^{1,p}(\Omega) \) into itself, provided \( v \) belong to the same \( \delta \)-ball.
We get from the boundedness of \( G \) that
\begin{equation*}
	\norm{T_{v}(r)}_{W^{1,p}(\Omega)} \leq C \norm{R(v+r)}_{W^{1-1/p',p'}(\partial\Omega)^{*}} \leq C \norm{R(v+r)}_{L^{p}(\Gamma_{I})}.
\end{equation*}
Since \( p>n \) the Sobolev embedding \( W^{1,p}(\Omega)\hookrightarrow C(\overline{\Omega}) \) is bounded and \( w,v,r \in C(\overline{\Omega})\).
By taking the supremum followed by using the Lipschitz property of \( \partial_{u}a(x,w) \) we get
\begin{equation}\label{eq-fixed-point-2}
\begin{aligned}
	 \norm{R(v+r)}_{L^{p}(\Omega)} &\leq C\norm{\partial_{u}a(x,w+t[v+r])-\partial_{u}a(x,w)}_{L^{\infty}(\Gamma_{I})}\norm{v+r}_{L^{\infty}(\Gamma_{I})} \\
	 &\leq C\norm{v+r}_{L^{\infty}(\Gamma_{I})}^{2} \\
	 &\leq C\norm{v+r}_{W^{1,p}(\Omega)}^{2}
\end{aligned}
\end{equation}
Since \( \norm{v}_{W^{1,p}(\Omega)}\leq \delta \) and \( \norm{r}_{W^{1,p}(\Omega)}\leq \delta \), it follows that
\begin{equation*}
	\norm{T_{v}(r)}_{W^{1,p}(\Omega)} \leq \delta
\end{equation*}
for small enough \( \delta \).

We next prove the contraction property.
For \( h_{1}=v+r_{1} \), \( h_{2}=v+r_{2} \) we have
\begin{equation*}
\begin{aligned}
	R(v+r_{1})-R(v+r_{2}) = &\int_{0}^{1}\big[\partial_ua(x,w+th_1)-\partial_ua(x,w)\big](h_1-h_2)\,dt \\
	&+ \int_{0}^{1}\big[\partial_ua(x,w+th_1) - \partial_ua(x,w+th_2)\big]h_2\,dt.
\end{aligned}
\end{equation*}
By using the Lipschitz property of \( \partial_{u}a \) we find that
\begin{equation*}
	\abs*{R(v+r_{1})-R(v+r_{2})} \leq C(\abs{v+r_{1}}+\abs{v+r_{2}})\abs{r_{1}-r_{2}}
\end{equation*}
uniformly on \( \Gamma_{I} \).
Using this, we estimate
\begin{equation*}
\begin{aligned}
	\norm{T_{v}(r_{1})-T_{v}(r_{2})}_{W^{1,p}(\Omega)} &\leq C\norm{R(v+r_{1})-R(v+r_{2})}_{L^{p}(\Gamma_{I})} \\
	&\leq C\norm{(\abs{v+r_1}+\abs{v+r_2})\abs{r_1-r_2}}_{L^{p}(\Gamma_{I})} \\
	&\leq C4\delta\norm{r_1-r_2}_{W^{1,p}(\Omega)} \\
	&\leq \frac{1}{2}\norm{r_1-r_2}_{W^{1,p}(\Omega)}
\end{aligned}
\end{equation*}
provided \( \delta \) is small enough.
So \( r\mapsto T_{v}(r) \) is indeed a contraction and Banach's fixed point theorem ensures the existence of a fixed point \( r = T_{v}(r) \).

Let \( r \) be the fixed point \( r= T_{v}(r) \) corresponding to \( v \).
It follows from \eqref{eq-fixed-point-2} that \( \norm{r}_{W^{1,p}(\Omega)}\leq C\norm{v+r}_{W^{1,p}(\Omega)}^{2} \).
By applying the triangle inequality and using the smallness of \( v \) and \( r \), all terms involving \( r \) can be absorbed on the left-hand side and \eqref{eq-fixed-point-quadratic} follows.

Next we prove that the map \( v\mapsto r \) is Lipschitz continuous.
Almost identical computations as above give the estimates
\begin{equation*}
\begin{aligned}
	\norm{r_{1}-r_{2}}_{W^{1,p}(\Omega)} &= \norm{T_{v_{1}}(r_{1}) - T_{v_{2}}(r_{2})}_{W^{1,p}(\Omega)} \\
	&\leq C\norm{R(v_{1}+r_{1})-R(v_{2}+r_{2})}_{L^{p}(\Gamma_{I})} \\
	&\leq C\norm{(\abs{v_{1}+r_{1}}+\abs{v_{2}+r_{2}})\abs{v_{1}+r_{1}-v_{2}-r_{2}}}_{L^{p}(\Gamma_{I})} \\
	&\leq C(\norm{v_{1}+r_{1}}_{L^{\infty}(\Gamma_{I})}+\norm{v_{2}+r_{2}}_{L^{\infty}(\Gamma_{I})})\norm{v_{1}+r_{1}-v_{2}-r_{2}}_{L^{\infty}(\Gamma_{I})}.
\end{aligned}
\end{equation*}
Using the boundedness of the Sobolev embedding and smallness of \( v,r \) results in
\begin{equation*}
	\norm{r_{1}-r_{2}}_{W^{1,p}(\Omega)} \leq 4C\delta(\norm{v_{1}-v_{2}}_{W^{1,p}(\Omega)}+\norm{r_{1}-r_{2}}_{W^{1,p}(\Omega)}).
\end{equation*}
Since \( \delta \) is small, the term \( \norm{r_{1}-r_{2}}_{W^{1,p}(\Omega)} \) can be absorbed on the left and we get
\begin{equation*}
	\norm{r_{1}-r_{2}}_{W^{1,p}(\Omega)} \leq C\norm{v_{1}-v_{2}}_{W^{1,p}(\Omega)}.\qedhere
\end{equation*}
\end{proof}

\begin{lemma}[Solution map]\label{lemma-solution-map}
Let \( n\geq 2 \), \( p>n \), and \( \Omega,\partial\Omega,\Gamma_{A},\Gamma_{I} \) satisfy \ref{assumption-1} and \ref{assumption-2}.
Assume that \( \gamma \) satisfies \ref{assumption-3} and \( a\in\mathcal{A}_{ad} \).
Let \( w \in W^{1,p}(\Omega) \) be a weak solution corresponding to Neumann data \( f=h \), for \( h\in W^{1-1/p',p'}(\partial\Omega)^{*} \) with \( \operatorname{supp}(h)\subseteq\Gamma_{A} \), of the equation
\begin{equation}\label{eq-nonlinear-solution-map}
\begin{cases}
	-\operatorname{div}(\gamma\nabla u) = 0&\text{in }\Omega, \\
	\partial_{\nu,\gamma}u = f&\text{on }\Gamma_{A}, \\
	\partial_{\nu,\gamma}u + a(x,u) = 0&\text{on }\Gamma_{I}.
\end{cases}
\end{equation}
For \( \delta>0 \) sufficiently small, let \( B_{\delta}\coloneqq \{v\in W^{1,p}(\Omega)\colon \norm{v}_{W^{1,p}(\Omega)}\leq \delta \} \) and let \( Q\colon B_{\delta}\mapsto W^{1,p}(\Omega) \) be the Lipschitz solution map obtained in Lemma \ref{lemma-fixed-point}.
Then \( S\colon B_{\delta}\to W^{1,p}(\Omega) \) given by \( S(v)\coloneqq w+v+Q(v) \) is a Lipschitz continuous map.
If, in addition, \( v \) is a weak solution of the equation
\begin{equation}\label{eq-linearized}
\begin{cases}
	-\operatorname{div}(\gamma\nabla v) = 0&\text{in }\Omega, \\
	\partial_{\nu,\gamma}v = g&\text{on }\Gamma_{A}, \\
	\partial_{\nu,\gamma}v + \partial_{u}a(x,w)v = 0&\text{on }\Gamma_{I},
\end{cases}
\end{equation}
for \( g\in W^{1-1/p',p'}(\partial\Omega)^{*} \) with \( \operatorname{supp}(g)\subseteq\Gamma_{A} \) then \( u = S(v) \) is a weak solution of \eqref{eq-nonlinear-solution-map} with \( f = h+g+\Phi(0,R(v+r)) \).

Conversely, for any solution \( u \) of \eqref{eq-nonlinear-solution-map} with \( \norm{u-w}_{W^{1,p}(\Omega)}\leq \tilde{\delta} \) for small enough \( \tilde{\delta}<\delta \) there exists a solution \( v\in B_{\delta} \) of \eqref{eq-linearized} with \( g\in W^{1-1/p',p'}(\partial\Omega)^{*} \) and \( \operatorname{supp}(g)\subseteq\Gamma_{A} \) such that \( u=w+v+Q(v) \).
\end{lemma}
\begin{proof}
Since Lemma \ref{lemma-fixed-point} establishes existence and Lipschitz continuity of the map \( Q(v) \), the map \( S(v) \) is well-defined and Lipschitz continuous.
Denote \( r = Q(v) \), then we need only verify that the function \( u = w+v+r \) satisfies the required weak form equation when \( v \) is a weak solution of \eqref{eq-linearized}.
Now \( w,v,r \) satisfy, for any \( \varphi\in W^{1,p'}(\Omega) \), the equations
\begin{equation*}
\begin{aligned}
	\int_{\Omega}\gamma\nabla w\cdot \nabla \varphi\,dx + \int_{\Gamma_{I}}a(x,w)\varphi\,ds &= \inner{h}{\varphi},\\%\int_{\Gamma_{A}}f_{w}\varphi\,ds
	\int_{\Omega}\gamma\nabla v\cdot \nabla \varphi\,dx + \int_{\Gamma_{I}}\partial_{u}a(x,w)v\varphi\,ds &= \inner{g}{\varphi},\\%\int_{\Gamma_{A}}f_{w}\varphi\,ds
	\int_{\Omega}\gamma\nabla r\cdot \nabla \varphi\,dx + \int_{\Gamma_{I}}[\partial_{u}a(x,w)r + R(v+r)]\varphi\,ds &= \int_{\Gamma_{A}}\Phi(v+r)\varphi\,ds.
\end{aligned}
\end{equation*}
By substituting \( u = w+v+r \) and using the first order Taylor expansion of \( a(x,w+v+r) \) at \( w \) in direction \( v+r \), we get
\begin{equation*}
\begin{aligned}
	\int_{\Omega}\gamma\nabla u\cdot \nabla\varphi\,dx &+ \int_{\Gamma_{I}}a(x,u)\varphi\,ds = \int_{\Omega}\gamma\nabla w\cdot \nabla\varphi\,dx + \int_{\Gamma_{I}}a(x,w)\varphi\,ds \\
	&+ \int_{\Omega}\gamma\nabla v\cdot \nabla\varphi\,dx + \int_{\Gamma_{I}}\partial_{u}a(x,w)v\varphi\,ds \\
	&+ \int_{\Omega}\gamma\nabla r\cdot \nabla\varphi\,dx + \int_{\Gamma_{I}}[\partial_{u}a(x,w)r + R(v+r)]\varphi\,ds \\
	&=\inner{h}{\varphi} + \inner{g}{\varphi} + \int_{\Gamma_{A}}\Phi(0,R(v+r))\varphi\,ds.
\end{aligned}
\end{equation*}
Hence, \( u \) is a weak solution of \eqref{eq-nonlinear-solution-map}.

To prove the converse, let \( u \) be any solution with \( \norm{w-u}_{W^{1,p}(\Omega)}\leq\tilde{\delta} \).
Define \( \psi \coloneqq P_{\mathcal{N}}(u-w) \) to be the projection onto the kernel \( \mathcal{N} \) (Recall \( P_{\mathcal{N}} \), \( P_{\mathcal{N}_{\Gamma_{A}}^{*}} \) from \eqref{eq-projections}).
Note that \( \norm{\psi}_{W^{1,p}(\Omega)}\leq C\tilde{\delta} \) since \( P_{\mathcal{N}} \) is bounded.
Let \( \varphi \) be the unique function satisfying \( \varphi\vert_{\Gamma_{A}}\in\mathcal{N}_{\Gamma_{A}}^{\perp} \) and solving
\begin{equation*}
\begin{cases}
    -\operatorname{div}(\gamma\nabla \varphi) = 0&\text{in }\Omega, \\
    \partial_{\nu,\gamma}\varphi = [\operatorname{Id}-P_{\mathcal{N}_{\Gamma_{A}}^{*}}](\partial_{\nu,\gamma}u-\partial_{\nu,\gamma}w)&\text{on }\Gamma_{A}, \\
    \partial_{\nu,\gamma}\varphi + \partial_{u}a(x,w)\varphi = 0&\text{on }\Gamma_{I}.
\end{cases}
\end{equation*}
The existence of \( \varphi \) is ensured by Lemma \ref{lemma-forced-well-posedness}.
It follows from the boundedness of \( P_{\mathcal{N}_{\Gamma_{A}}^{*}} \) and the boundedness of the conormal derivative \( u\mapsto \partial_{\nu,\gamma}u \) that \( \norm{\varphi}\leq C\tilde{\delta} \).
Let \( v\coloneqq \psi+\varphi \) and define \( r\coloneqq u-w-v \).
We require that \( r\vert_{\Gamma_{A}}\in\mathcal{N}_{\Gamma_{A}}^{\perp} \).
To prove this, first note that
\begin{equation}\label{eq-solution-map-projection}
	u-w-\psi = [\operatorname{Id}-P_{\mathcal{N}}](u-w)\in\mathcal{N}^{\perp}.
\end{equation}
Since \( u-w-\psi \) annihilate \( \mathcal{N} \) their trace onto \( \Gamma_{A} \) cannot belong to \( \mathcal{N}_{\Gamma_{A}} \), by Lemma \ref{lemma-kernel-basis}.
It then follows from \eqref{eq-decompositions} that
\begin{equation*}
	(u-w-\psi)\vert_{\Gamma_{A}}\in\mathcal{N}_{\Gamma_{A}}^{\perp}.
\end{equation*}
Since also \( \varphi\in\mathcal{N}_{\Gamma_{A}}^{\perp} \) we conclude that \( r\in\mathcal{N}_{\Gamma_{A}}^{\perp} \).
Before finding the weak form equation for \( r \), note that by using \( u = w+v+r \) and Taylor expansion around \( w \) we have
\begin{equation*}
\begin{aligned}
	a(x,u) - a(x,w)-\partial_{u}a(x,w)v &= a(x,w+v+r) - a(x,w)-\partial_{u}a(x,w)v \\
	&=\partial_{u}a(x,w)r + R(v+r).
\end{aligned}
\end{equation*}
From the definition of \( \psi,\varphi \) we have
\begin{equation*}
\begin{aligned}
	(\partial_{\nu,\gamma}u - \partial_{\nu,\gamma}w - \partial_{\nu,\gamma}v)\vert_{\Gamma_{A}} &= (\partial_{\nu,\gamma}u - \partial_{\nu,\gamma}w - \partial_{\nu,\gamma}\psi - \partial_{\nu,\gamma}\varphi)\vert_{\Gamma_{A}} \\
	&= (\partial_{\nu,\gamma}u - \partial_{\nu,\gamma}w - [\operatorname{Id}-P_{\mathcal{N}_{\Gamma_{A}}^{*}}](\partial_{\nu,\gamma}u - \partial_{\nu,\gamma}w))\vert_{\Gamma_{A}} \\
	&= P_{\mathcal{N}_{\Gamma_{A}}^{*}}(\partial_{\nu,\gamma}u - \partial_{\nu,\gamma}w).
\end{aligned}
\end{equation*}
Using the previous two identities together with the weak form equations for \( u,w,v \) we find the weak form equation for \( r \),
\begin{equation*}
\begin{aligned}
	\int_{\Omega} \gamma\nabla r\cdot \nabla\rho\,dx = &\int_{\Omega} \gamma\nabla (u-w-v)\cdot \nabla\rho\,dx \\
	=&-\int_{\Gamma_{I}}[a(x,u)-a(x,w)-\partial_{u}a(x,w)v]\rho\,ds \\
	&+\inner{\partial_{\nu,\gamma}u}{\rho} - \inner{\partial_{\nu,\gamma}w}{\rho}-\inner{\partial_{\nu,\gamma}v}{\rho} \\
	=&-\int_{\Gamma_{I}}[\partial_{u}a(x,w)r+R(v+r)]\rho\,ds \\
	&+\int_{\Gamma_{A}}P_{\mathcal{N}_{\Gamma_{A}}^{*}}(\partial_{\nu,\gamma}u - \partial_{\nu,\gamma}w)\rho\,ds \\
\end{aligned}
\end{equation*}
for any \( \rho\in W^{1,p'}(\Omega) \).
Hence \( r \) is a weak solution of \eqref{eq-fixed-point} with the Neumann data on \( \Gamma_{A} \) being \( P_{\mathcal{N}_{\Gamma_{A}}^{*}}(\partial_{\nu,\gamma}u - \partial_{\nu,\gamma}w)\in\mathcal{N}_{\Gamma_{A}} \).
From the uniqueness of the function \( \Phi \) in Lemma \ref{lemma-forced-well-posedness}, it follows that \( P_{\mathcal{N}_{\Gamma_{A}}^{*}}(\partial_{\nu,\gamma}u - \partial_{\nu,\gamma}w) = \Phi(0,R(v+r)) \).
Now \( r \) is a weak solution of
\begin{equation}\label{eq-solution-map-fixed-point}
\begin{cases}
	-\operatorname{div}(\gamma\nabla r) = 0&\text{in }\Omega, \\
	\partial_{\nu,\gamma}r = \Phi(0,R(v+r))&\text{on }\Gamma_{A}, \\
	\partial_{\nu,\gamma}r + \partial_{u}a(x,w)r = -R(v+r)&\text{on }\Gamma_{I}.
\end{cases}
\end{equation}
Recall that \( \norm{v}_{W^{1,p}(\Omega)}\leq C\tilde{\delta} \) for some \( C>0 \) which possibly depends on \( w \) and \( a \) but does not depend on \( u \).
It follows from the definition of \( r \) that \( \norm{r}_{W^{1,p}(\Omega)}\leq (1+C)\tilde{\delta} \).
For \( \tilde{\delta}<\frac{\delta}{1+C} \), this solution \( r \) belongs to the same small ball in which Lemma \ref{lemma-fixed-point} establishes existence and uniqueness of solution to \eqref{eq-solution-map-fixed-point}.
Hence \( r = Q(v) \) and the decomposition \( u = w+v+Q(v) \) is established.
\end{proof}

\section{The inverse problem}\label{sec-inverse-problem}

In this section we prove Theorem \ref{thm-main-result-1} and Theorem \ref{thm-main-result-2}.
The proofs are very similar to the proofs of the main results in \cite{jns23}.
Compared to \cite{jns23}, however, a couple of steps are simplified in the present setting and the conclusions are stronger.
In an attempt to highlight these differences, we present the proofs in full detail.
\begin{proof}[Proof of Theorem \ref{thm-main-result-1}]
From Lemma \ref{lemma-solution-map} we get solutions \( u_{1,v} = w_{1} + v + Q(v) \) for the equation with nonlinearity \( a_{1} \), provided \( \norm{v}_{W^{1,p}(\Omega)}<\tilde{\delta} \) for some \( \tilde{\delta}>0 \).
After possibly decreasing \( \tilde{\delta} \) the Cauchy data of \( u_{1,v} \) belong to \( C_{a_{1}}^{w_{1},\delta} \).
Using \( C_{a_{1}}^{w_{1},\delta} = C_{a_{2}}^{w_{2},C\delta} \) we find solutions \( u_{2,v} \) for the nonlinearity \( a_{2} \) and having Cauchy data equal to that of \( u_{1,v} \).
Then \( u_{1,v}-u_{2,v}\in W^{1,p}(\Omega) \) satisfies
\begin{equation*}
\begin{cases}
    -\operatorname{div}(\gamma\nabla (u_{1,v}-u_{2,v})) = 0&\text{in }\Omega, \\
    u_{1,v}-u_{2,v} = 0&\text{on }\Gamma_{A}, \\
    \partial_{\nu,\gamma}(u_{1,v}-u_{2,v}) = 0&\text{on }\Gamma_{A},
\end{cases}
\end{equation*}
and it follows by unique continuation (Lemma \ref{lemma-ucp}) that \( u_{1,v} = u_{2,v} \) on \( \overline{\Omega} \).
By subtracting the two weak form equations for \( a_{1},a_{2} \) and using the equality \( u_{1,v} = u_{2,v} \) we get
\begin{equation*}
	\int_{\Gamma_{I}}[a_{1}(x,u_{1,v}(x))-a_{2}(x,u_{1,v}(x))]\varphi(x)\,ds = 0 \quad\forall \varphi\in W^{1,p'}(\Omega),
\end{equation*}
and hence
\begin{equation}\label{eq-abstract-identification-1}
    a_{1}(x,u_{1,v}(x)) = a_{2}(x,u_{1,v}(x))\quad\text{for a.e. }x\in\Gamma_{I}.
\end{equation}
In other words, the nonlinearities are equal on the graphs of the solutions \( u_{1,v} \).

For any \( 0<\theta<1 \), Runge approximation (Theorem \ref{theorem-runge-w1p}) ensures the existence of Neumann data \( g \) and a weak solution \( v_{\theta} \) of
\begin{equation*}
\begin{cases}
    -\operatorname{div}(\gamma\nabla v_{\theta}) = 0&\text{in }\Omega, \\
    \partial_{\nu,\gamma}v_{\theta} = g&\text{on }\Gamma_{A}, \\
    \partial_{\nu,\gamma}v_{\theta} + \partial a_{1}(x,w)v_{\theta} =0&\text{on }\Gamma_{I},
\end{cases}
\end{equation*}
such that \( \norm{v_{\theta}-1}_{C(\Gamma_{I})}<\theta \).
We have in particular that
\begin{equation}\label{eq-identification-1}
	v_{\theta}(x)\geq 1-\theta \quad\forall x\in\Gamma_{I}.
\end{equation}
For \( \abs{t}<\tilde{\delta}/\norm{v_{\theta}}_{W^{1,p}(\Omega)} \), we get solutions \( u_{1,tv_{\theta}} = w_{1} + tv_{\theta} + Q(tv_{\theta}) \) satisfying
\begin{equation}\label{eq-identification-2}
	\abs{Q(tv_{\theta}(x))} \leq \norm{Q(tv_{\theta})}_{C(\overline{\Omega})} \leq \norm{Q(tv_{\theta})}_{W^{1,p}(\Omega)} \leq C\abs{t}^{2} \norm{v_{\theta}}_{W^{1,p}(\Omega)}^{2}.
\end{equation}
Let \( 0 < \beta < 1-\theta \), \( \tau \coloneqq \frac{1-\theta-\beta}{C\norm{v_{\theta}}_{W^{1,p}(\Omega)}^{2}} \) and \( \alpha\coloneqq\min\{\tau, \tilde{\delta}/\norm{v_{\theta}}_{W^{1,p}(\Omega)}\} \).
For \( \abs{t}<\alpha \), any value in a possibly small interval can be attained by some solution \( u_{1,tv_{\theta}} \).
To see this, let \( x\in\Gamma_{I} \) be arbitrary and consider the map \( [-\alpha,\alpha]\ni t\mapsto \eta(t)\coloneqq u_{1,tv_{\theta}}(x)-w_{1}(x) \).
For \( 0<t<\alpha \) it follows from \eqref{eq-identification-1} and \eqref{eq-identification-2} that
\begin{equation*}
\begin{aligned}
	\eta(t) = u_{1,tv_{\theta}}(x)-w_{1}(x) &= t v_{\theta}(x)+Q(t v_{\theta}(x)) \\
    &\geq t(1-\theta - tC\norm{v}_{W^{1,p}(\Omega)}^{2}) \\
    &\geq t\beta
\end{aligned}
\end{equation*}
and by replacing \( t \) with \( -t \) we get
\begin{equation*}
\begin{aligned}
	\eta(-t) = u_{1,-tv_{\theta}}(x)-w_{1}(x) &= -t v_{\theta}(x)+Q(-t v_{\theta}(x)) \\
    &\leq -t(1-\theta) + t^{2}C\norm{v}_{W^{1,p}(\Omega)}^{2} \\
    &= -t(1-\theta - tC\norm{v}_{W^{1,p}(\Omega)}^{2}) \\
    &\leq -t\beta.
\end{aligned}
\end{equation*}
In particular, \( \eta(\alpha)\geq \alpha\beta \) and \( \eta(-\alpha)\leq -\alpha\beta \).
Since \( \eta(t) \) is continuous it follows from the intermediate value theorem that \( \eta \) is surjective onto \( [-\alpha\beta,\alpha\beta] \).
This show that for any \( x\in\Gamma_{I} \) and any \( \abs{\varepsilon}\leq \lambda\coloneqq\alpha\beta \) there exist a solution \( u \) such that \( u(x) = w_{1}(x) + \varepsilon \).
Together with \eqref{eq-abstract-identification-1} this completes the proof.
\end{proof}

\begin{proof}[Proof of Theorem \ref{thm-main-result-2}]
Starting with any solution \( u_{1} \) corresponding to \( a_{1} \) we get from \( C_{a_{1}}=C_{a_{2}} \) a solution \( u_{2} \) with nonlinearity \( a_{2} \) and having equal Cauchy data.
As in the proof of Theorem \ref{thm-main-result-1} we conclude that \( u_{1} = u_{2} \) on \( \overline{\Omega} \) by unique continuation.
Since the solutions are arbitrary, it now follows that the reachable sets are equal,
\begin{equation*}
	E_{a_{1}} = E_{a_{2}}.
\end{equation*}
Proceeding as in the proof of Theorem \ref{thm-main-result-1}, it follows that
\begin{equation*}
    a_{1}(x,z) = a_{2}(x,z)\quad\text{for a.e. }(x,z)\in E_{a_{1}}.
\end{equation*}
In other words, the nonlinearities are equal on the reachable sets.

It remains to show that the reachable set is open.
Let \( (x',z')\in E_{a_{1}} \) be arbitrary and let \( w \) be a solution with \( w(x') = z' \).
As in the proof of Theorem \ref{thm-main-result-1} we may find a \( \lambda>0 \) such that
\begin{equation}\label{eq-identification-4}
	\{(x,w(x)+\varepsilon)\colon \abs{\varepsilon}<\lambda,x\in\Gamma_{I}\}\subseteq E_{a_{1}}.
\end{equation}
Let \( I = (z'-\frac{\lambda}{2},z'+\frac{\lambda}{2})\subseteq\mathbb{R} \).
Since \( w \) is continuous, there exist a ball \( B \) centered at \( x' \) with radius small enough that \( \abs{w(x)-w(x')}<\frac{\lambda}{2} \) for any \( x\in B \) and such that the intersection \( \mathcal{U}\coloneqq B\cap\partial\Omega \) is contained in \( \Gamma_{I} \).
Then \( \mathcal{U}\times I \) is a subset of the left-hand side of \eqref{eq-identification-4} and \( \mathcal{U}\times I \) is therefore a neighborhood of \( (x',z') \) in \( E_{a_{1}} \).
Hence \( E_{a_{1}} \) is open.
\end{proof}

\appendix
\section{Theory for the linear Robin problem}\label{sec-appendix}
While the focus of the article is a nonlinear Robin problem, the developed theory depends heavily on the theory of the linear Robin problem.
The linear theory may be considered standard but is scattered throughout the literature and may be difficult to find.
For the sake of completeness, the required \( W^{1,p} \) weak solution theory for the linear Robin problem is developed in this appendix.
We prove the Fredholm alternative in \( W^{1,p}(\Omega) \), give a rigorous interpretation of the conormal derivative of weak solutions, and establish Runge approximation and unique continuation from Cauchy data.

For \( 1<p<\infty \) and \( 0<s<1 \), we use Sobolev spaces defined as
\begin{equation*}
\begin{aligned}
	W^{1,p}(\Omega)&\coloneqq \{u\in L^{p}(\Omega)\colon \partial_{x_{j}}u\in L^{p}(\Omega)\,\,\text{ for }j\in\{1\ldots,n\} \}, \\
	W^{s,p}(\Gamma)&\coloneqq \{u\in L^{p}(\Gamma)\colon\seminorm{u}_{W^{s,p}(\Gamma)}<\infty\},
\end{aligned}
\end{equation*}
where \( \Gamma\subseteq\partial\Omega \) is a nonempty open subset and the Gagliardo seminorm is given by
\begin{equation*}
	\seminorm{u}_{W^{s,p}(\Gamma)}\coloneqq \Big(\int_{\Gamma}\int_{\Gamma}\frac{\abs{u(x)-u(y)}^{p}}{\abs{x-y}^{n-1+sp}}\,ds\,ds\Big)^{1/p}.
\end{equation*}
The usual norms are used,
\begin{equation*}
\begin{aligned}
	\norm{u}_{W^{1,p}(\Omega)}&\coloneqq \Big(\norm{u}_{L^{p}(\Omega)}^{p} + \sum_{\abs{\alpha}=1}\norm{\partial_{\alpha}u}_{L^{p}(\Omega)}^{p}\Big)^{1/p}, \\
	\norm{u}_{W^{s,p}(\Gamma)}&\coloneqq \Big(\norm{u}_{L^{p}(\Gamma)}^{p} + \seminorm{u}_{W^{s,p}(\Gamma)}^{p}\Big)^{1/p}.
\end{aligned}
\end{equation*}

Let \( f\in W^{1,p'}(\Omega)^{*} \), \( g,h\in W^{1-1/p',p'}(\partial\Omega)^{*} \), with \( \operatorname{supp}(g)\subseteq\Gamma_{A}, \operatorname{supp}(h)\subseteq\Gamma_{I} \) and consider the equation
\begin{equation}\label{eq-robin-linear}
\begin{cases}
	-\operatorname{div}(\gamma\nabla u) = f&\text{in }\Omega, \\
	\partial_{\nu,\gamma}u = g&\text{on }\Gamma_{A}, \\
	\partial_{\nu,\gamma}u + qu = h&\text{on }\Gamma_{I},
\end{cases}
\end{equation}
The weak form of this equation is 
\begin{equation*}
	\int_{\Omega}\gamma\nabla u\cdot\nabla\varphi\,dx + \int_{\Gamma_{I}}qu\varphi\,ds = \inner{f}{\varphi}+\inner{g}{\varphi\vert_{\partial\Omega}}+\inner{h}{\varphi\vert_{\partial\Omega}}\quad\forall\varphi\in W^{1,p'}(\Omega).
\end{equation*}
Note, however, that the functionals \( g,h \) are composed with the boundary trace, and these compositions are functionals in \( W^{1,p'}(\Omega)^{*} \).
It is therefore no loss in generality to replace the three functionals \( f,g,h \) by a single functional \( F\in W^{1,p'}(\Omega)^{*} \) in the weak form.
By doing this, we obtain the following weak form, which is the equation we study,
\begin{equation}\label{eq-robin-linear-weak}
	\int_{\Omega}\gamma\nabla u\cdot\nabla\varphi\,dx + \int_{\Gamma_{I}}qu\varphi\,ds = \inner{F}{\varphi}\quad\forall\varphi\in W^{1,p'}(\Omega).
\end{equation}
We rely on Fredholm theory and therefore need to understand the adjoint equation.
Consider the bilinear form \( B\colon W^{1,p}(\Omega)\times W^{1,p'}(\Omega)\to\mathbb{R} \) and the bounded linear map \( L\colon W^{1,p}(\Omega)\to W^{1,p'}(\Omega)^{*} \) defined by
\begin{equation}\label{eq-weak-operator}
\begin{aligned}
	B(u,v) &\coloneqq \int_{\Omega}\gamma\nabla u\cdot \nabla v\,dx + \int_{\Gamma_{I}}quv\,ds, \\
	[Lu](v)&\coloneqq B(u,v).
\end{aligned}
\end{equation}
The Banach space adjoint \( L^{*} \) is defined as a map \( W^{1,p'}(\Omega)^{**}\to W^{1,p}(\Omega)^{*} \).
The Sobolev spaces are reflexive, so we may view it as a map \( W^{1,p'}(\Omega)\to W^{1,p}(\Omega)^{*} \).
If we abuse notation and identify \( \varphi\in W^{1,p'}(\Omega) \) with its induced functional in \( W^{1,p'}(\Omega)^{**} \), then
\begin{equation*}
	[L^{*}\varphi](u) \coloneqq \varphi(Lu) = [Lu](\varphi).
\end{equation*}
By defining the adjoint bilinear form \( B^{*}(u,v) \coloneqq B(v,u) \), we have
\begin{equation*}
	[L^{*}\varphi](u) = B^{*}(\varphi,u).
\end{equation*}
Since the coefficient \( \gamma \) is symmetric, the adjoint operators \( L^{*} \) and \( B^{*} \) are given by the same formulas as \( L \) and \( B \) but with the roles of \( W^{1,p}(\Omega) \) and \( W^{1,p'}(\Omega) \) swapped.

The kernel of \( L \) and its trace space onto the boundary are denoted by
\begin{equation}\label{eq-kernels}
\begin{aligned}
	\mathcal{N} &\coloneqq \{u\in W^{1,p}(\Omega)\colon u \text{ solves } Lu = 0 \}, \\
	\mathcal{N}_{\Gamma_{A}} &\coloneqq \{u\vert_{\Gamma_{A}}\colon u\in \mathcal{N} \}.
\end{aligned}
\end{equation}

The conormal derivative is very important, but does not exist for general \( W^{1,p}(\Omega) \) functions.
By characterizing the conormal derivative as a functional, however, it can be generalized to the space
\begin{equation*}
\begin{aligned}
	E(L; L^{p}(\Omega)) \coloneqq \{&u\in W^{1,p}(\Omega)\colon \exists f\in L^{p}(\Omega) \\
	&\text{ s.t. } Lu=f\text{ in the sense of distributions}\}.
\end{aligned}
\end{equation*}
This is proved in Lemma \ref{lemma-conormal}.
All solutions used in the study of the inverse Robin problem belong to this space.

Let \( W^{1-1/p,p}_{0}(\partial\Omega,\Gamma) \) be the subset of functions in \( W^{1-1/p,p}(\partial\Omega) \) which vanish outside of \( \Gamma\subset\partial\Omega \),
\begin{equation}\label{eq-traces-space-support}
	W^{1-1/p,p}_{0}(\partial\Omega,\Gamma) = \{u\in W^{1-1/p,p}(\partial\Omega)\colon \operatorname{supp}(u)\subseteq\Gamma \}.
\end{equation}
The restriction of a functional \( f\in W^{1-1/p',p'}(\partial\Omega)^{*} \) to the subset \( \Gamma \) is the functional \( f\vert_{\Gamma} \in W^{1-1/p',p'}_{0}(\partial\Omega,\Gamma)^{*} \) defined by
\begin{equation}\label{eq-functional-restriction}
	\inner{f\vert_{\Gamma}}{\varphi} \coloneqq \inner{f}{\varphi}\quad\forall \varphi\in W^{1-1/p',p'}_{0}(\partial\Omega,\Gamma).
\end{equation}

\begin{lemma}
The space \( E(L;L^{p}(\Omega)) \) is a Banach space when equipped with the norm
\begin{equation*}
	\norm{u}_{E(L;L^{p}(\Omega))} \coloneqq \norm{u}_{W^{1,p}(\Omega)} + \norm{Lu}_{L^{p}(\Omega)}.
\end{equation*}
\end{lemma}
\begin{proof}
Let \( (u_{n})_{n\in\mathbb{N}}\subset E(L;L^{p}(\Omega)) \) be a Cauchy sequence, and let \( (f_{n})_{n\in\mathbb{N}}\subset L^{p}(\Omega) \) be the sequence of functions such that \( Lu_{n} = f_{n} \) in the sense of distributions.
Then \( f_{n} \) is Cauchy in \( L^{p}(\Omega) \) and \( u_{n} \) is Cauchy in \( W^{1,p}(\Omega) \), hence the limits \( f\in L^{p}(\Omega) \) and \( u\in W^{1,p}(\Omega) \) exist.
The convergence \( u_{n}\to u \) in \( W^{1,p}(\Omega) \) implies the convergence \( Lu_{n}\to Lu \) in \( W^{1,p'}(\Omega)^{*} \) by Hölder's inequality and we get in particular weak convergence \( Lu_{n}\rightharpoonup Lu \).
For any \( \varphi\in C_{c}^{\infty}(\Omega) \),
\begin{equation*}
	\int_{\Omega}f\varphi\,dx = \lim_{n}\int_{\Omega}f_{n}\varphi\,dx = \lim_{n} [Lu_{n}](\varphi) = [Lu](\varphi).
\end{equation*}
This shows that \( u\in E(L;L^{p}(\Omega)) \).
\end{proof}

\begin{lemma}[{\cite[Theorem 18.40]{leoni17}}]\label{lemma-trace}
Let \( 1<p<\infty \), \( n\geq 2 \) and \( \Omega\subset\mathbb{R}^{n} \) be an open bounded set with Lipschitz boundary \( \partial\Omega \).
Then any function \( u\in W^{1,p}(\Omega) \) satisfies
\begin{equation*}
	\norm{u}_{W^{1-1/p,p}(\partial\Omega)}\leq C\norm{u}_{W^{1,p}(\Omega)}
\end{equation*}
for some \( C>0 \) independent of \( u \).
Moreover, the Sobolev trace \( W^{1,p}(\Omega)\to W^{1-1/p,p}(\partial\Omega) \) has a bounded right inverse.
In other words, for any \( g\in W^{1-1/p,p}(\partial\Omega) \) there exists \( u\in W^{1,p}(\Omega) \) with \( u\vert_{\partial\Omega} = g \) such that
\begin{equation*}
	\norm{u}_{W^{1,p}(\Omega)}\leq C'\norm{g}_{W^{1-1/p,p}(\partial\Omega)}
\end{equation*}
for some \( C'>0 \) independent of \( g \) and \( u \).
\end{lemma}

\begin{lemma}[Conormal derivative]\label{lemma-conormal}
Let \( 1<p<\infty \), \( n\geq 2 \) and \( \Omega,\partial\Omega \) satisfy \ref{assumption-1} and \( \Gamma\subset\partial\Omega \) be a nonempty open subset.
Assume that \( \gamma \) satisfies \ref{assumption-3}.
Let \( L\colon W^{1,p}(\Omega)\to W^{1,p'}(\Omega)^{*} \) be the operator corresponding to \( -\operatorname{div}(\gamma\nabla u) \),
\begin{equation*}
	[Lu](\varphi) = \int_{\Omega}\gamma\nabla u\cdot\nabla \varphi\,dx.
\end{equation*}
The notion of conormal derivative can be extended from \( W^{2,p}(\Omega) \) to \( E(L; L^{p}(\Omega)) \) in the sense that if \( u\in E(L;L^{p}(\Omega)) \) then there is a unique functional in \( W^{1-1/p',p'}(\partial\Omega)^{*} \) denoted by \( \partial_{\nu,\gamma}u \) such that
\begin{equation}\label{eq-conormal-1}
	\inner{\partial_{\nu,\gamma}u}{\varphi} = \int_{\Omega}\gamma\nabla u\cdot\nabla e_{\varphi}-fe_{\varphi}\,dx
\end{equation}
where \( e_{\varphi}\in W^{1,p'}(\Omega) \) is any extension of \( \varphi \) into \( \Omega \) and \( f\in L^{p}(\Omega) \) is the function such that \( Lu = f \) in the sense of distributions.
\end{lemma}
\begin{proof}
Since \( u\in W^{1,p}(\Omega) \), we can integrate by parts and use the assumption that \( -\operatorname{div}(\gamma\nabla u) = f \) in the sense of distributions to get
\begin{equation*}
	\int_{\Omega}\gamma\nabla u\cdot\nabla\varphi\,dx = \int_{\Omega}u[-\operatorname{div}(\gamma\nabla\varphi)]\,dx = \int_{\Omega}f\varphi\,dx\quad\forall \varphi\in C_{c}^{\infty}(\Omega).
\end{equation*}
The first and last expressions belong to \( W^{1,p'}(\Omega)^{*} \).
By taking limits in \( W^{1,p'}(\Omega) \) we have
\begin{equation}\label{eq-conormal-2}
	\int_{\Omega}\gamma\nabla u\cdot\nabla\varphi\,dx = \int_{\Omega}f\varphi\,dx\quad\forall \varphi\in W^{1,p'}_{0}(\Omega).
\end{equation}
In other words, \( u \) is a weak solution of the Dirichlet problem
\begin{equation*}
\begin{cases}
	-\operatorname{div}(\gamma\nabla u) = f&\text{in }\Omega,\\
	u = g&\text{on }\partial\Omega,
\end{cases}
\end{equation*}
for \( g \coloneqq u\vert_{\partial\Omega}\in W^{1-1/p,p}(\partial\Omega) \).

Clearly the right-hand side of \eqref{eq-conormal-1} makes sense as a functional on \( e_{\varphi}\in W^{1,p'}(\Omega) \).
Let \( e_{\varphi},\tilde{e}_{\varphi}\in W^{1,p'}(\Omega) \) be two extensions of \( \varphi\in W^{1-1/p',p'}(\partial\Omega) \).
Then \( e_{\varphi}-\tilde{e}_{\varphi}\in W^{1,p'}_{0}(\Omega) \) and it follows from \eqref{eq-conormal-2} that
\begin{equation*}
	\int_{\Gamma}\gamma\nabla u\cdot\nabla(e_{\varphi}-\tilde{e}_{\varphi})\,dx - \int_{\Omega}f(e_{\varphi}-\tilde{e}_{\varphi})\,dx = 0,
\end{equation*}
which shows that \eqref{eq-conormal-1} is independent of the extension.
Boundedness follows by taking the extension \( e_{\varphi} \) to be given by the bounded right-inverse of the trace  (Lemma \ref{lemma-trace}).

Note that for \( u\in W^{2,p}(\Omega) \) the conormal derivative \( \partial_{\nu,\gamma}u \) exists as a function in \( W^{1-1/p,p}(\partial\Omega) \) and integration against this function gives a functional in \( W^{1-1/p',p'}(\partial\Omega)^{*} \),
\begin{equation*}
	\inner{\partial_{\nu,\gamma}u}{\varphi} = \int_{\partial\Omega}(\partial_{\nu,\gamma}u)\varphi\,ds.
\end{equation*}
The equality \( f = -\operatorname{div}(\gamma\nabla u) \) holds in \( L^{p}(\Omega) \), not only in the sense of distributions.
Integration by parts yields \eqref{eq-conormal-1}, which justifies the name conormal derivative.
\end{proof}

\begin{lemma}[Unique continuation from Cauchy data]\label{lemma-ucp}
Let \( n\geq 2 \) and \( \Omega,\partial\Omega \) satisfy \ref{assumption-1} and \( \Gamma\subset\partial\Omega \) be a nonempty open subset.
Assume that \( \gamma \) satisfies \ref{assumption-3}.
Let \( 1<p<\infty \).
If \( u\in W^{1,p}(\Omega) \) satisfies
\begin{enumerate}
	\item\label{ucp-1} \( -\operatorname{div}(\gamma\nabla u) = 0 \) in \( \Omega \) in the sense of distributions,
	\item\label{ucp-2} \( u\vert_{\Gamma} = 0 \),
	\item\label{ucp-3} \( \partial_{\nu,\gamma}u\vert_{\Gamma} = 0 \) in the sense of \eqref{eq-functional-restriction},
\end{enumerate}
then \( u = 0 \) on \( \overline{\Omega} \).
\end{lemma}
\begin{proof}
Since \( \gamma_{ij}\in C^{0,1}(\overline{\Omega}) \) are Lipschitz, they can be extended to \( \mathbb{R}^{n} \) as Lipschitz functions \( \tilde{\gamma}_{ij}\in C^{0,1}(\mathbb{R}^{n}) \).
We wish to extend the equation into a slightly larger domain.
Since \( \partial\Omega \) is \( C^{1} \), there exists a function \( \phi\colon\mathbb{R}^{n}\to\mathbb{R} \) such that
\begin{equation*}
	\Omega = \{x\in\mathbb{R}^{n}\colon \phi(x)>0\}, \quad\partial\Omega = \{x\in\mathbb{R}^{n}\colon \phi(x) = 0\},
\end{equation*}
and such that \( \nabla\phi\neq 0 \) on \( \partial\Omega \).
Let \( x_{0}\in \Gamma \) and \( \mathcal{U}\subset\mathbb{R}^{n} \) be a neighborhood of \( x_{0} \) such that \( \nabla\phi\neq 0 \) in \( \mathcal{U} \).
Let \( \mathcal{V}\subset\mathbb{R}^{n} \) be a neighborhood of \( \Omega \) in which \( \tilde{\gamma} \) is uniformly elliptic.
Let \( \mathcal{W}\subset\subset\mathcal{Q}\subset\subset\mathcal{U}\cap\mathcal{V} \) be two neighborhoods of \( x_{0} \) and take \( \psi\in C_{c}^{\infty}(\mathbb{R}^{n}) \) satisfying \( 0\leq\psi\leq 1 \), \( \operatorname{supp}(\psi)\subseteq \mathcal{Q} \), and \( \psi\vert_{\mathcal{W}} = 1 \).
Let \( \varepsilon = \inf_{x\in\mathcal{U}}\{\abs{\nabla\phi (x)}\} \) and \( \delta = \sup_{x\in\mathcal{U}}\{\abs{\nabla\psi(x)}\} \).
Then \( \nabla(\phi+\frac{\varepsilon}{2\delta}\psi)\neq 0 \) in \( \mathcal{U} \).
Define
\begin{equation*}
	\Omega_{0} = \{x\in\mathbb{R}^{n}\colon \phi(x)+\frac{\varepsilon}{2\delta}\psi(x) > 0 \}.
\end{equation*}
Since \( \phi = 0 \) on \( \partial\Omega \) and \( \psi\vert_{\mathcal{W}} = 1 \), it follows that \( \phi(x)+\frac{\varepsilon}{2\delta}\psi(x)>0 \) in a neighborhood of \( x_{0} \).
Hence \( \Omega\subseteq \Omega_{0} \).
Additionally, \( \nabla(\phi+\frac{\varepsilon}{2\delta}\psi)\neq 0 \) on \( \partial\Omega_{0} \), and \( \Omega_{0} \) is therefore a \( C^{1} \) domain.
Note that \( \Omega_{0} \) is connected, \( \partial\Omega\cap\Omega_{0}\subseteq\Gamma \), and the operator \( -\operatorname{div}(\tilde{\gamma}\nabla u) \) is uniformly elliptic in \( \Omega_{0} \).

Extend the solution \( u \) to
\begin{equation*}
	v(x) =
\begin{cases}
	u(x)&x\in\Omega,\\
	0&x\in\Omega_{0}\setminus\Omega.
\end{cases}
\end{equation*}
Then \( v\in W^{1,p}(\Omega_{0}) \) since \( u\vert_{\Gamma} = 0 \).
We claim that \( v \) satisfies
\begin{equation}\label{ucp-4}
	\int_{\Omega_{0}}\tilde{\gamma}\nabla v\cdot \nabla \varphi\,dx = 0\quad\forall\varphi\in C_{c}^{\infty}(\Omega_{0}).
\end{equation}
If the test function \( \varphi \) is supported in \( \Omega \) then \eqref{ucp-4} hold because of \eqref{ucp-1}.
When \( \varphi \) is supported in \( \Omega_{0}\setminus\Omega \) then \eqref{ucp-4} is true since \( v\vert_{\Omega_{0}\setminus\Omega} = 0 \).
If the support of \( \varphi \) intersects the boundary section \( \Gamma \) then \( \psi = \varphi\vert_{\overline{\Omega}} \) belong to \( C^{\infty}(\overline{\Omega}) \) and has \( \operatorname{supp}(\psi\vert_{\partial\Omega})\subseteq\Gamma \).
Then it follows from \eqref{ucp-3} that
\begin{equation*}
	\int_{\Omega_{0}}\tilde{\gamma}\nabla v\cdot\nabla\varphi\,dx = \int_{\Omega}\gamma\nabla u\cdot\nabla\psi\,dx = \inner{\partial_{\nu,\gamma}u}{\psi} = 0.
\end{equation*}
So \eqref{ucp-4} is indeed valid.
Note that when \( p<2 \), interior regularity \cite[Proposition A.1]{jls17} implies that \( v\in W^{1,2}_{loc}(\Omega_{0}) \).
It now follows from Theorem 2.4 in \cite{hörmander83} that \( v = 0 \) on \( \overline{\Omega}_{0} \) and hence also \( u = 0 \) on \( \overline{\Omega} \).
\end{proof}

\begin{lemma}[Friedrichs' inequality]\label{lemma-friedrich}
Let \( \Omega\subset \mathbb{R}^{n} \), \( n\geq 2 \), be a bounded open set with Lipschitz boundary \( \partial\Omega \) and let \( \Gamma\subset \partial\Omega \) be a subset with positive measure.
Then there exists a constant \( C > 0 \) depending only on \( \Omega \) such that any \( u \in W^{1,2}(\Omega) \) satisfies
\begin{equation*}
	\norm{u}_{L^{2}(\Omega)} \leq C(\norm{\nabla u}_{L^{2}(\Omega)} + \norm{u}_{L^{2}(\Gamma)}).
\end{equation*}
\end{lemma}
\begin{proof}
The proof is by contradiction, similar to the proof of the standard Poincaré inequality \cite[Theorem 1, Chapter 5.8]{evans10}. Suppose to the contrary that for every \( n\in\mathbb{N} \) there exists \( u_n\in W^{1,2}(\Omega) \) such that
\begin{equation}\label{eq-friedrich-1}
	\norm{\nabla u_n}_{L^{2}(\Omega)} + \norm{u_n}_{L^{2}(\Gamma)} \leq \frac{1}{n}\norm{u_n}_{L^{2}(\Omega)}.
\end{equation}
If we normalize the sequence \( u_n \) so that \( \norm{u_n}_{L^2(\Omega)} = 1 \) then the sequence is bounded in \( W^{1,2}(\Omega) \) and therefore has a subsequence (also denoted \( u_{n} \)) which converges weakly in \( W^{1,2}(\Omega) \) to some function \( u\in W^{1,2}(\Omega) \).
From \eqref{eq-friedrich-1} it follows that the partial derivatives \( \partial_{\alpha}u_{n} \) converge strongly in \( L^{2}(\Omega) \) to \( 0 \), and hence also converges weakly to \( 0 \).
By uniqueness of weak limit in \( W^{1,2}(\Omega) \) it follows that \( \nabla u = 0 \) and \( u \) is therefore constant.
Next, by compactness, \( u_{n}\vert_{\Gamma} \) converges strongly to \( u\vert_{\Gamma} \) in \( L^{2}(\Gamma) \).
From \eqref{eq-friedrich-1} we see that \( u_{n}\vert_{\Gamma} \) converges to \( 0 \).
By uniqueness of limit, \( u\vert_{\Gamma} = 0 \) and then \( u = 0 \) on \( \Omega \) since it is constant.
Finally, \( u_{n} \) converges strongly to \( u \) in \( L^{2}(\Omega) \) by compactness and since \( \norm{u_{n}}_{L^{2}(\Omega)} = 1 \) we also have \( \norm{u}_{L^{2}(\Omega)} = 1 \), but this is a contradiction.
\end{proof}

\begin{lemma}[Boundary regularity]\label{lemma-boundary-regularity}
Let \( \Omega\subset \mathbb{R}^{n} \), \( n\geq 2 \), be a bounded open set with \( C^{1} \) boundary \( \partial\Omega \) and let \( \Gamma\subset \partial\Omega \) be a subset with positive measure.
Assume that \( \gamma \) satisfies \ref{assumption-3} and \( q\in L^{\infty}(\Gamma) \).
Let \( F\in W^{1,p'}(\Omega)^{*} \), \( p\geq 2 \), and suppose that \( u\in W^{1,2}(\Omega) \) satisfies
\begin{equation*}
	[Lu](\varphi) = \int_{\Omega}\gamma\nabla u\cdot\nabla\varphi\,dx + \int_{\Gamma}qu\varphi\,ds = F(\varphi)\quad\varphi\in W^{1,p'}(\Omega).
\end{equation*}
Then \( u \in W^{1,p}(\Omega) \).
\end{lemma}
\begin{proof}
We establish the result by a reduction to interior regularity \cite[Proposition A.1]{jls17}, by straightening the boundary and reflecting across the straight part of the boundary.
~\newline
{\bfseries Step 1 - Localization:}
Let \( \{B_{x_{i}}\}_{i=1}^{N} \) be an open cover of \( \partial\Omega \) of small enough sets that each boundary section \( \partial\Omega\cap B_{x_{i}} \) can be flattened.
Let \( \{ \psi_{i} \}_{i=1}^{N} \) be a smooth partition of unity subordinate to \( \{B_{x_{i}}\}_{i=1}^{N} \).
We show that \( \psi_{i}u\in W^{1,p}(\Omega) \) for each \( i \), from which the same conclusion follows for \( u = \sum_{i=1}^{N}\psi_{i}u \).
By using the test function \( \varphi\psi_{i} \) for an arbitrary \( \varphi \) we obtain the equation for \( \psi_{i}u \),
\begin{equation}\label{eq-boundary-regularity-1}
	\int_{\Omega}\gamma\nabla(\psi_{i}u)\cdot\nabla \varphi\,dx = \inner{G}{\varphi}\quad\forall\varphi\in W^{1,2}(\Omega)
\end{equation}
where
\begin{equation}\label{eq-robin-linear-regularity-2}
\begin{aligned}
	\inner{G}{\varphi} &\coloneqq \inner{F}{\psi_{i}\varphi} - \int_{\Gamma}q\psi_{i}u\varphi\,ds+ \int_{\Omega}\gamma u\nabla\psi_{i}\cdot\nabla\varphi - \gamma\varphi\nabla u\cdot\nabla\psi_{i}\,dx.
\end{aligned}
\end{equation}
~\newline
{\bfseries Step 2 - Flattening the boundary:}
Since the boundary is \( C^{1} \) there exists coordinates \( \Phi\colon B_{x_{i}}\cap\Omega\to V_{i}\in\mathbb{R}^{n} \) belonging to \( C^{1} \) which flattens the boundary.
From the precise definition of these coordinates (see e.g. \cite[Appendix C]{evans10}) it follows that
\begin{equation*}
	\operatorname{det}(D\Phi) = \operatorname{det}(D[\Phi^{-1}]) = 1.
\end{equation*}
We apply the change of variables \( x\mapsto\Phi(x) \) to \eqref{eq-boundary-regularity-1} in order to obtain an equation for \( v\coloneqq (\psi_{i}u)\circ \Phi^{-1} \).
First note that
\begin{equation*}
	\partial_{i}[v\circ\Phi](x) = \sum_{k}\partial_{k}v(y)\vert_{y=\Phi(x)}\partial_{i}\Phi_{k}(x).
\end{equation*}
The terms \( \partial_{i}\Phi_{k}(x) \) are incorporated into the coefficient \( a(x) = a_{ij}(x) \) in the transformed domain,
\begin{equation*}
	a_{ij}(x) = \sum_{k}\sum_{l}[\gamma_{kl}(y)\partial_{l}\Phi_{j}(y)\partial_{k}\Phi_{i}(y)]\vert_{y=\Phi^{-1}(x)}
\end{equation*}
Then
\begin{equation*}
	\int_{B_{x_{i}}\cap\Omega}\gamma\nabla (\psi_{i}u)\cdot\nabla\varphi\,dx = \int_{V_{i}\cap\mathbb{R}^{n}_{+}}a_{ij}(x)\nabla v\cdot\nabla \tilde{\varphi}\,dx,
\end{equation*}
In a similar way, the functional \( G \) is transformed into a functional on \( W^{1,p'}(V_{i}\cap\mathbb{R}^{n}_{+}) \).
The equation \eqref{eq-boundary-regularity-1} is transformed into
\begin{equation}\label{eq-robin-linear-regularity-1}
	\int_{V_{i}\cap\mathbb{R}^{n}_{+}} a_{ij}(x) \nabla v\cdot\nabla \varphi\,dx = \inner{G}{\varphi}\quad\forall\varphi\in W^{1,p'}(V_{i}\cap\mathbb{R}^{n}_{+}).
\end{equation}

The coefficients \( a_{ij} \) are uniformly continuous.
In order to apply elliptic regularity, we additionally need to verify uniform ellipticity.
Let \( \xi\in\mathbb{R}^{n} \) and \( \eta^{T}=\xi^{T}D\Phi \).
Then \( \eta_{i}=\sum_{k}\partial_{i}\Phi_{k}\xi_{k} \) and
\begin{equation*}
\begin{aligned}
	\xi^{T}a(x)\xi &= \sum_{i}\sum_{j}\sum_{k}\sum_{l} [\gamma_{kl}(y)\partial_{l}\Phi_{j}(y)\partial_{k}\Phi_{i}(y)]\vert_{y=\Phi^{-1}(x)}\xi_{i}\xi_{j} \\
	&= \sum_{k}\sum_{l}\gamma_{kl}(y)\vert_{y=\Phi^{-1}(x)}\Big(\sum_{i}\partial_{k}\Phi_{i}(y)\vert_{y=\Phi^{-1}(x)}\xi_{i}\Big)\Big(\sum_{j} \partial_{l}\Phi_{j}(y)\vert_{y=\Phi^{-1}(x)}\xi_{j}\Big) \\
	&= \sum_{k}\sum_{l}\gamma_{kl}(y)\vert_{y=\Phi^{-1}(x)}\eta_{k}\eta_{l}
	\geq C\norm{\eta}_{\mathbb{R}^{n}}^{2}
\end{aligned}
\end{equation*}
From \( \Phi(\Phi^{-1}(x)) = x \) it follows that \( D\Phi D[\Phi^{-1}] = \operatorname{Id} \).
Therefore, \( \xi=\eta^{T}D[\Phi^{-1}] \) and \( \norm{\xi}_{\mathbb{R}^{n}}\leq C\norm{\eta}_{\mathbb{R}^{n}} \).
This implies that \( \xi^{T}a(x)\xi\geq C\norm{\xi}_{\mathbb{R}^{n}}^{2} \).
~\newline
{\bfseries Step 3 - Reflection across the boundary:}
We extend by even reflection every function appearing in \eqref{eq-robin-linear-regularity-1}.
More precisely, let \( \tilde{V}_{i} = \{(x',-x_{n})\colon (x',x_{n})\in V_{i}\cap \mathbb{R}^{n}_{+}\} \) denote the extended domain and extend each function by the formula
\begin{equation*}
	v(x',x_{n}) = v(x',-x_{n}) \quad\text{for }(x',x_{n})\in\tilde{V}_{i}.
\end{equation*}
Note that the coefficient \( a_{ij} \) extended in this fashion is uniformly continuous.
Applying the change of variables \( (x',x_{n})\to(x',-x_{n}) \) in \eqref{eq-robin-linear-regularity-1} leads to the same equation but in \( \tilde{V}_{i} \).
By adding the two equations we get
\begin{equation}\label{eq-boundary-regularity-2}
	\int_{V_{i}\cup \tilde{V}_{i}} a_{ij}(x) \nabla v\cdot\nabla \varphi\,dx = \inner{G}{\varphi}\quad\forall\varphi\in W^{1,p'}(V_{i}\cup \tilde{V}_{i}).
\end{equation}
Note that the operator \( G \) is of the same form as in \eqref{eq-robin-linear-regularity-2}, except that some coefficients are multiplied by the \( L^{\infty} \) functions \( \partial_{k}\Phi_{l} \).
~\newline
{\bfseries Step 4 - Improved regularity:}
We proceed by applying interior regularity \cite[Proposition A.1]{jls17} to \eqref{eq-boundary-regularity-2}.
Let us write the functional \( G \) as
\begin{equation*}
\begin{aligned}
	\inner{G}{\varphi} &\coloneqq \inner{F}{\psi_{i}\varphi} - \int_{\{x_{n}=0\}\cap V_{i}}b\varphi\,ds+ \int_{V_{i}\cup\tilde{V}_{i}}c\cdot\nabla\varphi - d\varphi\,dx.
\end{aligned}
\end{equation*}
where \( b = q\psi_{i}u \), \( c=\gamma u\nabla\psi_{i} \) and \( d = \gamma\nabla u\cdot\nabla\psi_{i} \).
Since \( b \), \( c \) depend linearly on \( u \), we can increase the \( L^{p} \) regularity of \( b \), \( c \) by applying the Sobolev embedding on \( u \).
Since \( d \) depends on \( \nabla u \), we instead increase its regularity by applying the Sobolev embedding to \( \varphi \).
In this way, we prove that \( G \) is a functional in a better space than originally defined.

Let \( p_{1}=2 \).
By the Sobolev embedding, \( u\in L^{s} \) with \( s = \frac{p_{1}n}{n-p_{1}} = \frac{2n}{n-2} \) and \( u\vert_{\{x_{n}=0\}\cap V_{i}}\in L^{r} \) with \( r = p_{1}\frac{n-1}{n-p_{1}} = 2\frac{n-1}{n-2} \).
Integrating \( c\cdot\nabla\varphi \) therefore induces a functional on \( \varphi\in W^{1,s'} \), where \( s' = \frac{np_{1}}{n(p_{1}-1)+p_{1}} = \frac{2n}{n+2} \) is the Hölder conjugate exponent.
Similarly, integrating \( b\varphi \) induces a functional on \( L^{r'}(\{x_{n}=0\}\cap V_{i}) \), which in turn is a functional on \( W^{1,t} \) for any \( t \) such that the trace of \( W^{1,t} \) embeds into \( L^{r'}(\{x_{n}=0\}\cap V_{i}) \).
The smallest such \( t \) ensured by the Sobolev embedding is \( t = \frac{np_{1}}{n(p_{1}-1)+p_{1}} = s' \).
For the term \( d\varphi \), we only have \( d\in L^{2} = L^{p_{1}} \).
We apply the Sobolev embedding to find that integrating \( d\varphi \) induces a functional on \( \varphi\in W^{1,\alpha'} \) where \( 2=p_{1}'=\frac{n\alpha'}{n+\alpha'} \).
By solving for \( \alpha' \) and substituting \( p_{1}'=\frac{p_{1}}{p_{1}-1} \), we find \( \alpha' = \frac{np_{1}'}{n+p_{1}'} = \frac{np_{1}}{n(p_{1}-1)+p_{1}}=s' \).
In other words, \( b,c,d \) all induce functionals on the same space \( W^{1,s'} \).
By invoking interior regularity, we conclude that \( v\in W^{1,p_{2}} \) where \( p_{2}=\min\{s,p\} \).
Repeating this for every localized function \( \psi_{i}u \) yields \( u\in W^{1,p_{2}}(\Omega) \).

Replacing \( p_{1} \) by \( p_{2} \) in the above, we get an increasing sequence \( p_{k} \) given by \( p_{k} = \frac{np_{1}}{n-(k-1)p_{1}} = \frac{2n}{n-(k-1)2} \).
Eventually, either \( p_{k}\geq n \) in which case the Sobolev embedding gives unrestricted exponent \( s \) in \( L^{s} \), or we get \( p_{k}>p \).
In both cases \( G \) is a functional on \( W^{1,p'} \), so that \( \psi_{i}u\in W^{1,p} \) by interior regularity.
Hence, \( u\in W^{1,p}(\Omega) \).
\end{proof}

\begin{lemma}\label{lemma-kernel-basis}
Let \( n\geq 2 \) and \( \Omega,\partial\Omega,\Gamma_{A},\Gamma_{I} \) satisfy \ref{assumption-1} and \ref{assumption-2}.
Assume that \( \gamma \) satisfies \ref{assumption-3} and \( q\in L^{\infty}(\Gamma_{I}) \).
Let \( \mathcal{N},\mathcal{N}_{\Gamma_{A}} \) be given by \eqref{eq-kernels}.
Then \( \psi_{1},\ldots,\psi_{N} \) form a basis of \( \mathcal{N} \) if and only if \( \psi_{1}\vert_{\Gamma_{A}},\ldots,\psi_{N}\vert_{\Gamma_{A}} \) form a basis of \( \mathcal{N}_{\Gamma_{A}} \).
In particular, if \( \mathcal{N} \) is nontrivial then there exists a basis of \( \mathcal{N} \) whose traces onto \( \Gamma_{A} \) are \( L^{2}(\Gamma_{A}) \)-orthonormal.
\begin{proof}
First suppose that \( \{\psi_{1},\ldots, \psi_{N}\} \) is a basis of \( \mathcal{N} \).
Then \( \mathcal{N}_{\Gamma_{A}} \) is a subspace of the linear span of \( \{\psi_{1}\vert_{\Gamma_{A}},\ldots,\psi_{N}\vert_{\Gamma_{A}}\} \).
If \( \mathcal{N}_{\Gamma_{A}} \) is a strict subspace then some \( \psi_{i} \) and \( \psi_{j} \) have the same trace on \( \Gamma_{A} \) and then it follows by Lemma \ref{lemma-ucp} that \( \psi_{i}=\psi_{j} \) on \( \overline{\Omega} \) which contradicts that they are part of a basis of \( \mathcal{N} \).
For the other direction, suppose that \( \psi_{1},\ldots,\psi_{N} \) form a basis of \( \mathcal{N}_{\Gamma_{A}} \).
By definition of \( \mathcal{N}_{\Gamma_{A}} \), \( \psi_{1},\ldots,\psi_{N} \) can be extended to functions in \( \mathcal{N} \).
If \( \psi_{1},\ldots,\psi_{N} \) are not a basis of \( \mathcal{N} \), then there exist a function \( \varphi\in\mathcal{N} \) such that \( \{\varphi,\psi_{1},\ldots,\psi_{N}\} \) are linearly independent in \( \mathcal{N} \).
But since \( \psi_{1},\ldots,\psi_{N} \) form a basis of \( \mathcal{N}_{\Gamma_{A}} \), \( \varphi\vert_{\Gamma_{A}} \) is a linear combination of \( \psi_{1}\vert_{\Gamma_{A}},\ldots,\psi_{N}\vert_{\Gamma_{A}} \).
It follows by Lemma \ref{lemma-ucp} that \( \varphi \) is a linear combination of \( \psi_{1},\ldots,\psi_{N} \) in \( \mathcal{N} \) too.

Finally, by taking a basis of \( \mathcal{N}_{\Gamma_{A}} \) which is \( L^{2}(\Gamma_{A}) \)-orthonormal and applying the first part of the lemma, we obtain a basis of \( \mathcal{N} \) with the desired orthonormality.
\end{proof}
\end{lemma}

\begin{lemma}\label{lemma-kernel-projection}
Let \( p, n\geq 2 \) and \( \Omega,\partial\Omega,\Gamma_{A},\Gamma_{I} \) satisfy \ref{assumption-1} and \ref{assumption-2}.
Assume that \( \gamma \) satisfies \ref{assumption-3} and \( q\in L^{\infty}(\Gamma_{I}) \).
Then
\begin{equation}\label{lemma-projection-decompositions}
\begin{aligned}
	L^{p}(\Omega) &= \mathcal{N}\oplus\mathcal{N}^{\perp}, \\
	L^{p}(\Gamma_{A}) &= \mathcal{N}_{\Gamma_{A}}\oplus\mathcal{N}_{\Gamma_{A}}^{\perp_{L}},\\
	W^{1-1/p',p'}_{0}(\partial\Omega,\Gamma_{A})^{*} &= \mathcal{N}_{\Gamma_{A}} \oplus \mathcal{N}_{\Gamma_{A}}^{\perp_{W}}, \\
\end{aligned}
\end{equation}
where \( \mathcal{N}_{\Gamma_{A}}^{\perp_{L}} \), \( \mathcal{N}_{\Gamma_{A}}^{\perp_{W}} \) denotes the subspaces of \( L^{p}(\Gamma_{A}) \), \( W^{1-1/p',p'}_{0}(\partial\Omega,\Gamma_{A})^{*} \), respectively, that annihiliate \( \mathcal{N}_{\Gamma_{A}} \).
Moreover, there are projections \( P_{\mathcal{N}} \), \( P_{\mathcal{N}_{\Gamma_{A}}} \), \( P_{\mathcal{N}_{\Gamma_{A}}^{*}} \) in \( L^{p}(\Omega) \), \( L^{p}(\Gamma_{A}) \), \( W^{1-1/p',p'}_{0}(\partial\Omega, \Gamma_{A})^{*} \) onto \( \mathcal{N} \), \( \mathcal{N}_{\Gamma_{A}}\subset L^{p}(\Gamma_{A}) \), \( \mathcal{N}_{\Gamma_{A}}\subset W^{1-1/p',p'}_{0}(\partial\Omega, \Gamma_{A})^{*} \), respectively, given by
\begin{equation*}
\begin{aligned}
	P_{\mathcal{N}}u &= \sum_{i=1}^{N}\inner{u}{\psi_{i}}_{L^{2}(\Omega)}\psi_{i}, \\
	P_{\mathcal{N}_{\Gamma_{A}}}v &= \sum_{i=1}^{N}\inner{v}{\tilde{\psi}_{i}}_{L^{2}(\Gamma_{A})}\tilde{\psi}_{i}\vert_{\Gamma_{A}}, \\
	P_{\mathcal{N}_{\Gamma_{A}}^{*}}f &= \sum_{i=1}^{N}\inner{f}{\tilde{\psi}_{i}}\tilde{\psi}_{i}\vert_{\Gamma_{A}},
\end{aligned}
\end{equation*}
where \( \psi_{1},\ldots,\psi_{N} \) is an \( L^{2}(\Omega) \)-orthonormal basis of the kernel \( \mathcal{N} \) and \( \tilde{\psi}_{1},\ldots,\tilde{\psi}_{N} \) an \( L^{2}(\Gamma_{A}) \)-orthonormal basis of \( \mathcal{N}_{\Gamma_{A}} \).
\begin{proof}
First note that, for any \( 2<p<\infty \), Lemma \ref{lemma-boundary-regularity} implies that the kernel \( \mathcal{N} \) and its trace space \( \mathcal{N}_{\Gamma_{A}} \) are subspaces of \( L^{p}(\Omega) \) and \( L^{p}(\Gamma_{A}) \), respectively.
Since the sums are finite, the operators are clearly well-defined.
Boundedness of \( P_{\mathcal{N}} \) and \( P_{\mathcal{N}_{\Gamma_{A}}} \) follows from Hölder's inequality and boundedness of \( P_{\mathcal{N}_{\Gamma_{A}}^{*}} \) follows from the boundedness of \( f \).
The fact that \( P_{\mathcal{N}}^{2} = P_{\mathcal{N}}\) follows easily by verifying that \( P_{\mathcal{N}}u = u \) for arbitrary \( u=\sum_{i=1}\alpha_{i}\psi_{i}\in\mathcal{N} \).
Similarly, \( P_{\mathcal{N}_{\Gamma_{A}}}^{2} = P_{\mathcal{N}_{\Gamma_{A}}} \), and \( P_{\mathcal{N}_{\Gamma_{A}}^{*}}^{2} = P_{\mathcal{N}_{\Gamma_{A}}^{*}} \).
The decompositions \eqref{lemma-projection-decompositions} now follow from Theorem 13.2 in Chapter III in \cite{conway19}.
\end{proof}
\end{lemma}

\begin{lemma}\label{lemma-linear-solvability}
Let \( n\geq 2 \), \( 1<p<\infty \), and \( \Omega,\partial\Omega,\Gamma_{A},\Gamma_{I} \) satisfy \ref{assumption-1} and \ref{assumption-2}.
Assume that \( \gamma \) satisfies \ref{assumption-3} and \( q\in L^{\infty}(\Gamma_{I}) \).
Then there exists for any \( F\in W^{1,p'}(\Omega)^{*} \) a solution \( u\in W^{1,p}(\Omega) \) of the weak form equation
\begin{equation*}
	[Lu](\varphi) = \int_{\Omega}\gamma\nabla u\cdot\nabla\varphi\,dx + \int_{\Gamma_{I}}qu\varphi\,ds = F(\varphi)\quad\varphi\in W^{1,p'}(\Omega),
\end{equation*}
if and only if
\begin{equation}\label{eq-compat-2}
	F(\varphi) = 0\quad\forall\varphi\in\operatorname{Ker}(L^{*}).
\end{equation}
Moreover, there is a unique solution with \( u\vert_{\Gamma_{A}}\in \mathcal{N}_{\Gamma_{A}}^{\perp} \) and this solution satisfies
\begin{equation*}
	\norm{u}_{W^{1,p}(\Omega)} \leq C\norm{F}_{W^{1,p'}(\Omega)^{*}}.
\end{equation*}
\begin{proof}
~\newline
{\bfseries Step 1 - Perturbed problem, \( p = 2 \):}
Consider the perturbed problem
\begin{equation*}
	[L_{\lambda}u](\varphi) \coloneqq \int_{\Omega}\gamma\nabla u\cdot\nabla\varphi\,dx + \int_{\Gamma_{I}}(q+\lambda)u\varphi\,ds = F(\varphi)\quad\varphi\in W^{1,p'}(\Omega),
\end{equation*}
where \( \lambda>\norm{q}_{L^{\infty}(\Gamma_{I})} \).
Define \( B_{\lambda}\colon W^{1,2}(\Omega)\times W^{1,2}(\Omega)\to\mathbb{R} \) by \( B_{\lambda}(u,\varphi)\coloneqq [L_{\lambda}u](\varphi) \).
By using Lemma \ref{lemma-friedrich}, the positivity of \( q+\lambda \) and the uniform ellipticity we find that \( B_{\lambda} \) is coercive,
\begin{equation*}
\begin{aligned}
	\norm{u}_{W^{1,2}(\Omega)}^{2} &\leq C_{1}(\norm{u}_{L^{2}(\Gamma_{I})}^{2} +\norm{\nabla u}_{L^{2}(\Omega)}^{2}) \\
	&\leq C_{2}B_{\lambda}(u,u).
\end{aligned}
\end{equation*}
From the trace inequality it follows that \( B_{\lambda} \) is bounded.
Therefore, Lax-Milgram's lemma ensures, for any \( F\in W^{1,2}(\Omega)^{*} \), existence and uniqueness of solution \( u\in W^{1,2}(\Omega) \) of
\begin{equation*}
	B_{\lambda}(u,\varphi) = \inner{F}{\varphi}\quad\forall \varphi\in W^{1,2}(\Omega).
\end{equation*}
~\newline
{\bfseries Step 2 - Perturbed problem, \( p\neq 2 \):}
When \( p>2 \) then \( W^{1,p'}(\Omega)^{*}\subset W^{1,2}(\Omega)^{*} \).
From Step 1 we get a weak solution \( u\in W^{1,2}(\Omega) \), and Lemma \ref{lemma-boundary-regularity} implies that \( u\in W^{1,p}(\Omega) \).
Since \( W^{1,p}(\Omega)\subset W^{1,2}(\Omega) \), uniqueness of solutions in \( W^{1,2}(\Omega) \) implies uniqueness in \( W^{1,p}(\Omega) \).
Hence the perturbed problem is well-posed for \( p>2 \).

When \( 1<p<2 \), we use a standard adjoint argument.
Let \( L_{\lambda,p} \) denote the operator \( L_{\lambda}\colon W^{1,p}(\Omega)\to W^{1,p'}(\Omega)^{*} \) for a fixed \( p \).
Then \( L_{\lambda,p'} \) is invertible for \( p'\geq 2 \), and this implies invertibility of its adjoint \( L_{\lambda,p'}^{*}\colon W^{1,p}(\Omega)\to W^{1,p'}(\Omega)^{*} \).
Since the operator is symmetric, we have \( L_{\lambda,p'}^{*} = L_{\lambda,p} \).
~\newline
{\bfseries Step 3 - Compact reformulation:}
We now seek a solution \( u\in W^{1,p}(\Omega) \) of the equation \( Lu = F \) for \( F\in W^{1,p'}(\Omega)^{*} \) and \( 1<p<\infty \).
Let \( \iota\colon W^{1,p}(\Omega)\to W^{1,p'}(\Omega)^{*} \) denote the operator \( [\iota u](\varphi)\coloneqq \int_{\Gamma_{I}}u\varphi\,ds \).
Then \( \iota \) is compact, by the compactness of the trace \( W^{1,p}(\Omega)\to L^{p}(\Gamma_{I}) \).
By adding and subtracting \( \lambda \iota u \) to the equation and using the invertibility of \( L_{\lambda} \) results in
\begin{equation*}
	F = L_{\lambda}u -\lambda \iota u = L_{\lambda}(\operatorname{Id}-\lambda L_{\lambda}^{-1}\iota)u.
\end{equation*}
Applying \( L_{\lambda}^{-1} \) on both sides show that the equation \( Lu = F \) is equivalent to
\begin{equation*}
	(\operatorname{Id}-\lambda L_{\lambda}^{-1}\iota)u = L_{\lambda}^{-1}(F).
\end{equation*}
Since the operator \( \lambda L_{\lambda}^{-1}\iota \) is compact, the Fredholm alternative is valid and solutions therefore exist if and only if
\begin{equation}\label{eq-compat-1}
	L_{\lambda}^{-1}(F)\in\operatorname{Ker}([\operatorname{Id}-\lambda L_{\lambda}^{-1}\iota]^{*})^{\perp}.
\end{equation}
To see that \eqref{eq-compat-1} is equivalent to \eqref{eq-compat-2}, first note that for \( v\in W^{1,p'}(\Omega) \), \( \varphi\in W^{1,p}(\Omega) \),
\begin{equation*}
\begin{aligned}
	\inner{[\operatorname{Id}-\lambda L_{\lambda}^{-1}\iota]^{*}L_{\lambda}^{*}v}{\varphi} &= \inner{L_{\lambda}^{*}v}{[\operatorname{Id}-\lambda L_{\lambda}^{-1}\iota]\varphi} \\
	&= \inner{v}{[L_{\lambda}-\lambda\iota]\varphi} \\
	&= \inner{v}{L\varphi} \\
	&= \inner{L^{*}v}{\varphi}.
\end{aligned}
\end{equation*}
From this we conclude that \( v\in\operatorname{Ker}(L^{*}) \) if and only if \( L_{\lambda}^{*}v\in\operatorname{Ker}([\operatorname{Id}-\lambda L_{\lambda}^{-1}\iota]^{*}) \).
By using this with the invertibility of \( L_{\lambda}^{*} \) we see that any element in \( \operatorname{Ker}([\operatorname{Id}-\lambda L_{\lambda}^{-1}\iota]^{*}) \) is of the form \( L_{\lambda}^{*}v \) for some \( v\in\operatorname{Ker}(L^{*}) \).
Therefore, \eqref{eq-compat-1} holds if and only if \( \inner{L_{\lambda}^{-1}F}{L_{\lambda}^{*}v} = 0 \) for any \( v\in\operatorname{Ker}(L^{*}) \).
By writing \( \inner{F}{v} = \inner{L_{\lambda}L_{\lambda}^{-1}F}{v} = \inner{L_{\lambda}^{-1}F}{L_{\lambda}^{*}v} \) for \( v\in\operatorname{Ker}(L^{*}) \) we conclude that \eqref{eq-compat-1} is equivalent to \eqref{eq-compat-2}.
~\newline
{\bfseries Step 4 - Uniqueness:}
To establish existence of solutions with trace in \( \mathcal{N}_{\Gamma_{A}}^{\perp} \), let \( u \) be any solution with data \( F \).
Because of decompositions \eqref{lemma-projection-decompositions} there exists functions \( h\in\mathcal{N} \), \( k\in \mathcal{N}_{\Gamma_{A}}^{\perp} \) such that \( u\vert_{\Gamma_{A}} = h\vert_{\Gamma_{A}} + k \).
Then \( u-h \) is another solution and its trace is \( k\in N_{\Gamma_{A}}^{\perp} \).
This proves existence.

To establish uniqueness of solution with trace in \( \mathcal{N}_{\Gamma_{A}}^{\perp} \), suppose that \( u_{1},u_{2} \) are any two solutions with \( u_{i}\vert_{\Gamma_{A}}\in\mathcal{N}_{\Gamma_{A}}^{\perp} \) corresponding to the data \( F \).
Then \( (u_{1}-u_{2})\vert_{\Gamma_{A}}\in\mathcal{N}_{\Gamma_{A}}^{\perp} \).
On the other hand, \( u_{1}-u_{2}\in\mathcal{N} \) so that \( (u_{1}-u_{2})\vert_{\Gamma_{A}}\in\mathcal{N}_{\Gamma_{A}} \) and we necessarily have \( (u_{1}-u_{2})\vert_{\Gamma_{A}} = 0 \).
Then \( u_{1}-u_{2} \) is a weak solution of
\begin{equation*}
\begin{cases}
	-\operatorname{div}(\gamma\nabla(u_{1}-u_{2})) = 0&\text{in }\Omega,\\
	u_{1}-u_{2} = 0&\text{on }\Gamma_{A},\\
	\partial_{\nu,\gamma}(u_{1}-u_{2}) = 0&\text{on }\Gamma_{A},\\
	\partial_{\nu,\gamma}(u_{1}-u_{2}) + q(u_{1}-u_{2}) = 0&\text{on }\Gamma_{I}.
\end{cases}
\end{equation*}
It follows from unique continuation (Lemma \ref{lemma-ucp}) that \( u_{1} = u_{2} \) on \( \overline{\Omega} \).
~\newline
{\bfseries Step 5 - Stability:}
Let \( X\coloneqq\{u\in W^{1,p}(\Omega)\colon u\vert_{\Gamma_{A}}\in\mathcal{N}_{\Gamma_{A}}^{\perp}\} \).
The previous steps in the proof show that the operator \( L\colon X\mapsto W^{1,p'}(\Omega)^{*} \) is bounded and a bijection onto its range.
It follows from the Fredholm compatibility condition that \( \operatorname{Ran}(L) = \operatorname{Ker}(L^{*})^{\perp} \) and the closed range theorem \cite[Section 5 of Chapter VII]{yoshida95} that the range is closed.
The open mapping theorem now ensures the boundedness of \( L^{-1}\colon\operatorname{Ker}(L^{*})^{\perp}\to X \),
\begin{equation*}
	\norm{u}_{W^{1,p}(\Omega)} = \norm{L^{-1}F}_{W^{1,p}(\Omega)} \leq C\norm{F}_{W^{1,p'}(\Omega)^{*}}.\qedhere
\end{equation*}
\end{proof}
\end{lemma}

\begin{theorem}[Runge approximation]\label{theorem-runge-w1p}
Let \( n\geq 2 \), \( p>n \), and \( \Omega,\partial\Omega,\Gamma_{A},\Gamma_{I} \) satisfy \ref{assumption-1} and \ref{assumption-2}.
Assume that \( \gamma \) satisfies \ref{assumption-3} and \( q\in L^{\infty}(\Gamma_{I}) \).
Consider the equation
\begin{equation}\label{eq-runge-w1p-1}
\begin{cases}
    -\operatorname{div}(\gamma\nabla u) = 0&\text{in }\Omega, \\
    \partial_{\nu,\gamma}u = f&\text{on }\Gamma_{A}, \\
    \partial_{\nu,\gamma}u + qu = 0&\text{on }\Gamma_{I},
\end{cases}
\end{equation}
and let the space \( V \) be defined by
\begin{equation*}
\begin{aligned}
	V\coloneqq \{ u\vert_{\Gamma_{I}}\colon &u\in W^{1,p}(\Omega) \text{ solves \eqref{eq-runge-w1p-1} for some } f\in W^{1-1/p',p'}(\partial\Omega)^{*} \\
	&\text{with }\operatorname{supp}(f)\subseteq\Gamma_{A} \}.
\end{aligned}
\end{equation*}
Then the space \( V \) is dense in \( C(\Gamma_{I}) \),
\begin{equation*}
    \overline{V} = C(\Gamma_{I}).
\end{equation*}
\end{theorem}
\begin{proof}
By Hahn-Banach, it is enough to show that any Radon measure in \( C(\Gamma_{I})^{*} \) which vanishes on \( V \) must vanish on \( C(\Gamma_{I}) \).
Let \( h \) be such a measure,
\begin{equation}\label{eq-runge-w1p-2}
	\int_{\Gamma_{I}}u\,dh = 0\quad\forall u\in V.
\end{equation}
Since \( p>n \), the trace \( W^{1,p}(\Omega)\to C(\Gamma_{I}) \) is bounded by the Sobolev embedding and \( C(\Gamma_{I})^{*}\subset W^{1,p}(\Omega)^{*} \) so that \( h\in W^{1,p}(\Omega)^{*} \).
Let \( w\in W^{1,p'}(\Omega) \) be the unique weak solution with \( w\vert_{\Gamma_{A}}\in \mathcal{N}_{\Gamma_{A}}^{\perp} \), provided by Lemma \ref{lemma-linear-solvability}, of the adjoint problem,
\begin{equation*}
\begin{cases}
    -\operatorname{div}(\gamma\nabla w) = 0&\text{in }\Omega, \\
    \partial_{\nu,\gamma}w = 0&\text{on }\Gamma_{A}, \\
    \partial_{\nu,\gamma}w + qw = h&\text{on }\Gamma_{I}.
\end{cases}
\end{equation*}
Then \( w \) satisfies the weak form equation
\begin{equation}\label{eq-runge-w1p-3}
	\int_{\Omega}\gamma\nabla w\cdot \nabla \varphi\,dx + \int_{\Gamma_{I}}qw\varphi\,ds = \int_{\Gamma_{I}}\varphi\,dh\quad\forall \varphi\in W^{1,p}(\Omega).
\end{equation}
Since \( u\in V \) is a valid test function, we find, by taking arbitrary \( \varphi=u\in V \), that
\begin{equation}\label{eq-runge-w1p-4}
	\inner{f}{w} = \int_{\Omega}\gamma\nabla u\cdot \nabla w\,dx + \int_{\Gamma_{I}}quw\,ds = \int_{\Gamma_{I}}u\,dh = 0.
\end{equation}
Since \( f \in \mathcal{N}^{\perp} \) is arbitrary, we may take \( f \) to be the functional \( \psi\mapsto\int_{\Gamma_{A}}w\psi\,ds \).
Then \eqref{eq-runge-w1p-4} reduces to \( \norm{w}_{L^{2}(\Gamma_{A})}^{2} = 0 \), so that \( w\vert_{\Gamma_{A}} = 0 \).
Now \( w \) satisfies
\begin{equation*}
\begin{cases}
    -\operatorname{div}(\gamma\nabla w) = 0&\text{in }\Omega, \\
    w = 0&\text{on }\Gamma_{A}, \\
    \partial_{\nu,\gamma}w = 0&\text{on }\Gamma_{A}.
\end{cases}
\end{equation*}
It follows from unique continuation (Lemma \ref{lemma-ucp}) that \( w = 0 \) on \( \overline{\Omega} \) and \eqref{eq-runge-w1p-3} reduces to
\begin{equation*}
	\int_{\Gamma_{I}}\varphi\,dh = 0\quad\forall\varphi\in W^{1,p}(\Omega).
\end{equation*}
Since \( W^{1,p}(\Omega)\subset C(\overline{\Omega}) \) is dense in \( C(\overline{\Omega}) \), it follows by continuity of the functional \( h \) that
\begin{equation*}
	\int_{\Gamma_{I}}\varphi\,dh = 0\quad\forall \varphi\in C(\Gamma_{I}).\qedhere
\end{equation*}
\end{proof}

\printbibliography
\end{document}